\documentclass[12pt, a4paper,leqno,oneside]{amsart}
\usepackage{color}
\usepackage{amsmath}
\usepackage{amssymb}

\usepackage{tikz}
\usepackage{pgfplots}
\usepackage{enumerate}

\usepackage{amsthm}

\usepackage{graphicx}
\usepackage{courier}

\usepackage[colorlinks=true, pdfstartview=FitV, linkcolor=blue, citecolor=blue, urlcolor=blue,pagebackref=false]{hyperref}

\usepackage{bbm}

\usepackage{psfrag}

\usepackage{wasysym}

\usepackage{fullpage}

\usepackage{type1cm}

\newtheorem{theorem}{Theorem}[section]
\newtheorem{corollary}[theorem]{Corollary}
\newtheorem{proposition}[theorem]{Proposition}
\newtheorem{lemma}[theorem]{Lemma}

\numberwithin{equation}{section}

\theoremstyle{definition}
\newtheorem{definition}[theorem]{Definition}
\theoremstyle{remark}
\newtheorem{remark}[theorem]{Remark}

\renewcommand{\leq}{\leqslant}                   
\renewcommand{\geq}{\geqslant}                   

\DeclareMathOperator\Exp{Exp}

\def\d{\mathrm{d}}
\def\m{m}

\def\Z{\mathbb{Z}}
\def\R{\mathbb{R}}

\newcommand{\1}[1]{{\mathbbm{1}}_{#1}}

\usepackage{titlesec}
\titleformat{\section}{\Large\bfseries}{\thesection}{1em}{}
\titleformat{\subsection}{\bfseries}{\thesubsection}{1em}{}

\usepackage{array}
\newcolumntype{e}{>{\displaystyle}r @{\,} >{\displaystyle}c @{\,} >{\displaystyle}l}

  \newcounter{constant}
  \newcommand{\newconstant}[1]{\refstepcounter{constant}\label{#1}}
  \newcommand{\useconstant}[1]{c_{\textnormal{\tiny \ref{#1}}}}
\def\clap#1{\hbox to 0pt{\hss#1\hss}}

\def\mathclap{\mathpalette\mathclapinternal}

\def\mathclapinternal#1#2{\clap{$\mathsurround=0pt#1{#2}$}}

\usepackage{calc}

\def\arraypar#1{\parbox[c]{\textwidth - 2cm}{\centering #1}}

\usepackage{environ}
\NewEnviron{display}{
\begin{equation}\begin{array}{c}
  \arraypar{\BODY}
\end{array}\end{equation}
}

\newcommand{\mcup}{\textstyle \bigcup\limits}
\newcommand{\motimes}{\textstyle \bigotimes\limits}

\usepackage{xfrac}

\title{\mbox{} \\ \LARGE \usefont{T1}{tnr}{b}{n} \selectfont A\lowercase{bsorbing-state} \lowercase{transition} \lowercase{for} S\lowercase{tochastic} S\lowercase{andpiles} \lowercase{and} A\lowercase{ctivated} R\lowercase{andom} W\lowercase{alks}\\
\mbox{}
}

\author{\large \itshape S\lowercase{idoravicius, }V.$^1$ \itshape \qquad \quad T\lowercase{eixeira, }A.$^1$ \\ \mbox{}}
\address{$^1$ Instituto Nacional de Matem\'atica Pura e Aplicada --
IMPA,\newline Estrada Dona Castorina 110, 22460-320, Rio de Janeiro RJ, Brazil\newline
e-mail: {\itshape \texttt{augusto@impa.br}}\newline
e-mail: {\itshape \texttt{vladas@impa.br}}
}

\date{\today}

\begin{document}
\begin{abstract}
  We study the dynamics of two conservative lattice gas models on the infinite $d$-dimensional hypercubic lattice: the Activated Random Walks (ARW) and the Stochastic Sandpiles Model (SSM), introduced in the physics literature in the early nineties.
  Theoretical arguments and numerical analysis predicted that the ARW and SSM undergo a phase transition between an \emph{absorbing phase} and an \emph{active phase} as the initial density crosses a critical threshold.
  However a rigorous proof of the existence of an absorbing phase was known only for one-dimensional systems.
  In the present work we establish the existence of such phase transition in any dimension.
  Moreover, we obtain several quantitative bounds for how fast the activity ceases at a given site or on a finite system.
  The multi-scale analysis developed here can be extended to other contexts providing an efficient tool to study non-equilibrium phase transitions.
\end{abstract}

\maketitle

Mathematics Subject Classification (2000) 60K35, 82C20, 82C22, 82C26

\renewcommand\footnotemark{}
\renewcommand\footnoterule{}
\let\thefootnote\relax

\section{Introduction}

Models of avalanches in sandpiles introduced in \cite{MR949160} became a paradigm example of a phenomenon called self-organized criticality.
In contrast with typical statistical mechanics systems, models of this type are not equipped with a tuning parameter such as temperature, for which a phase transition is observed.
Instead, they are expected to naturally drive themselves to a critical equilibrium distribution, featuring power law statistics.
For a mathematical overview of this model see \cite{MR2581895}.

Since its introduction, the subject has attracted much attention and several related models have been studied (e.g. Abelian Networks, Stochastic Sandpiles, Activated Random Walkers, etc.), see for instance \cite{0305-4470-24-7-009}, \cite{MR2220576} and \cite{2013arXiv1309.3445B}.
These particle systems share some similar features and, in particular, present a delicate balance between \emph{activation} and \emph{stabilization}  of particles.

To be more concrete, let us briefly describe one of the models we analyze here: the Stochastic Sandpile Model (SSM).
The SSM is continuous time particle system on a finite or infinite graph, where every site $x$ contains a number of particles $\eta(x)$.
Let $\kappa > 0$ be a fixed integer. If $\eta(x) < \kappa$,   we say that the site is stable.
But whenever the number of particles reaches $\kappa$, the site $x$ is said to be unstable and after an exponential time it sends $\kappa$ particles to its neighboring sites uniformly at random.
This toppling rule is endowed with the mechanisms of activation and stabilization that we mentioned above.

\vspace{4mm}

{\bf Finite systems -} The original models of self organized criticality were defined on finite lattices and this requires the introduction of two additional ingredients, called \emph{driving} and \emph{dissipation}.
For sake of concreteness, suppose that the model is defined on a finite box of $\mathbb{Z}^d$.
Then, whenever a particle is sent to the exterior of the box, it is eliminated from the system.
This is called  \emph{dissipation}.
Under this rule, since the box is finite, the system reaches an absorbing state after a finite time.
At this point we add another particle uniformly over the box, a procedure called \emph{driving}.
The addition of this new particle could of course destabilize the chosen site, and the toppling of this site could destabilize some of its neighbors and so on, until a stable configuration is reached again.

The chain reaction that may be provoked by a single toppling is what causes the system to be so interesting.
These avalanches are believed to display power law distributions at equilibrium and this motivates the claim that they display a critical behavior. See \cite{MR1044086}, \cite{MR2581895} for a proof of such behavior on some specific examples.

\vspace{4mm}

{\bf Infinite systems -} In an attempt to explain how driving/dissipation models (DDM) drift themselves naturally towards criticality, Dickman, Mu\~noz, Vespignani, and Zapperi introduced in \cite{2001AIPC..574..102M} a corresponding conservative lattice gas model, called Fixed Energy Sandpile (FES).

The FES model is defined on the infinite lattice $\mathbb{Z}^d$, starting with an i.i.d. initial distribution of particles, with density $\zeta > 0$.
However, in the FES model the number of particles is conserved, meaning that the driving and dissipation mechanisms are not present.
The system evolves in continuous time, with unstable sites toppling at constant rate.
Simulations from \cite{2001AIPC..574..102M} indicate that the FES model undergoes a phase transition at a critical density $\zeta_c \in (0, \infty)$, that is ``for energy densities below $\zeta_c$, the FES-Sandpile falls into an \emph{absorbing state}, while for values above that threshold activity is indefinitely sustained, i.e.
the FES is in an \emph{active phase}.''

As we said, the main motivation behind the introduction of the FES model was to better understand the DDM sandpiles.
The reasoning behind this goes as follows.
For DDM models, particles are gradually introduced into the system, but while the density inside the box remains below $\zeta_c$, the system is in an absorbing phase and will accommodate those extra particles in its bulk, effectively increasing its density.
On the other hand, as long as the system's density is above $\zeta_c$, activity would be sustained and the dissipation at the boundary will bring the density down.
This explanation found support for instance in \cite{MR2220576}, Ch. 15.4.5. and \cite{MR2150148}.

The proposed relation between FES and DDM systems predicts in particular that $\zeta_c$ coincides with the asymptotic density $\zeta_s$, observed on finite systems at equilibrium, as the size of the box grows.
This has been verified for some exactly solvable models in $d = 1$, see \cite{MR2150148}.
However some recent works (both numerical \cite{2010PhRvL.104n5703F} and rigorous for $d = 1$ \cite{2012arXiv1211.4760F}) have found a discrepancy between these values in some models.
This indicates that the relation between FES and DDM models may be less direct, suggesting that issues such as non-ergodicity and toppling invariants \cite{2006EPJB...52...91C} may play an important role in this connection.
See also \cite{PhysRevE.84.066119} and \cite{2014arXiv1402.3283L} for a detailed description of this relation in the case of Abelian Sandpiles.
It is also important to mention that stochastic and deterministic models seem to belong to different universality classes, see \cite{2006PhRvL..96e8003D}.

\vspace{4mm}

{\bf Main results -} From a rigorous perspective, we are still far from understanding the relation between FES models and self organized criticality for stochastic models.
In fact, very little is known even about basic properties of FES models.
This article attempts to fill some of these gaps, most notably, the existence of a phase transition for some important models.

Before stating our main results, let us describe another particle system that fits into our scope and plays a central role in the present work, the so called Activated Random Walks (ARW) model. All of our results hold for both the ARW and the SSM models and could in principle be adapted to other processes, see Section~\ref{s:other_models}.

The ARW model is a special case of a class introduced by F. Spitzer in the 1970's and has since then been intensively studied, see \cite{DRS09} and references therein. The model can be described as a reaction-diffusion system, where particles can be of type $A$ or $B$ and are subject to the following transitions. Particles of type $A$ perform simple random walks with jump rates $D_A = 1$, while $B$-type particles do not move $D_B = 0$. An $A$-particle turns into a $B$-particle ($A \to B$ at rate $\lambda > 0$). On the other hand, whenever an $A$-particle jumps onto a $B$ particle, they become immediately active ($A + B \to 2A$ at infinity rate).

The first contribution of the current paper is to show that
\begin{equation}
  \begin{array}{c}
  \text{the ARW on the infinite lattice undergoes a phase transition}\\
  \text{at some $\zeta_c \in (0,\infty)$ and the same also holds true for the SSM,}
\end{array}
\end{equation}
see Theorem~\ref{t:absorption}.
To be more precise, our work shows the existence of an absorbing phase (by showing that for $\zeta$ small enough the system ceases activity), as conjectured in \cite{DRS09} and \cite{2012InMat.188..127R}. The fact that $\zeta_c < \infty$ (existence of an active phase for large $\zeta$) was proved for $d = 1$ in \cite{2012InMat.188..127R} and later for every $d \geq 1$ in \cite{MR2651824}.

Although FES seem closer to a traditional statistical mechanics problem, where a tuning parameter controls the phase transition, there are important differences that make the former harder to analyze.
The main being that FES models are intrinsically out of equilibrium.
Tools for dealing with non-equilibrium phase transitions are still incipient.

The techniques we develop to prove the above results are robust and give a clear picture of how absorption is achieved. In particular, our proofs work
\begin{itemize}
\item for different models, see Section~\ref{s:other_models},
\item for any dimension $d \geq 1$ and
\item they provide quantitative results, as described in \eqref{e:short_tails} and \eqref{e:few_exit}.
\end{itemize}
Moreover, the proofs have plenty of room for improvement and we believe that following works will elucidate these systems even further.

Let us mention some of the previous works that have been done on this subject. Some partial results in this direction have been obtained for $d = 1$  in \cite{MR2150148}, Proposition 2.2.
Existence of the absorbing phase for $d=1$ was established in \cite{2012InMat.188..127R}.
Although the proof of this one dimensional case is very instructive and elegant, it has serious obstructions to be generalized to higher dimensions.

\vspace{4mm}

As we mentioned above, the techniques developed in the current paper go beyond the mere existence of a phase transition for FES models. To illustrate this fact, we show two quantitative results concerning the \emph{expected time of absorption} and the \emph{warm-up phase of DDM models}.

Let $\tau$ be the last time of activity for the origin in either the ARW or SSM on the infinite lattice. Then for $\zeta$ small enough, we show that
\begin{equation}
  \label{e:fast_decay}
  P_\zeta [\tau > l] \leq c \exp\{ -c \log(l)^2\}, \text{ for every $l \geq 0$,}
\end{equation}
see Theorem~\ref{t:absorption}.

This result motivates the definition of
\begin{equation}
  \zeta_* = \sup_\zeta \{ E_\zeta(\tau) < \infty \},
\end{equation}
which can be seen as a strengthening of the definition of $\zeta_c$.
From \eqref{e:fast_decay}, one sees that $\zeta_* > 0$, leaving open the question of whether it equals $\zeta_c$, in analogy with the case of percolation, see \cite{Men86}.

Another interesting quantitative statement that we establish concerns the DDM type models away from equilibrium. For this, consider either the ARW or SSM on a finite box with side length $n$ in the presence of driving and dissipation at the boundary. Then, if we start from an empty configuration and if $\zeta$ is small enough,
\begin{equation}
  \begin{array}{c}
  \text{after the insertion of $\zeta n^d$ particles, no more than $n^{d - 1 + o(1)}$}\\
  \text{of them get dissipated in the boundary.}
  \end{array}
\end{equation}
See Theorem~\ref{t:dd} for a more precise statement of the above. In Figure~\ref{f:warmup} we illustrate the above result with a simulation.

\begin{figure}[ht]
  \label{f:warmup}
  \centering
  \begin{tikzpicture}[scale=1]
  \begin{axis}[title=Simulation for a $30 \times 30$ box,xlabel=Particles introduced, ylabel=Particles remaining,font=\tiny]
  \addplot[smooth] coordinates {
    (10  ,10) (20  ,20) (30  ,30) (40  ,40) (50  ,50) (60  ,60) (70  ,70) (80  ,80) (90  ,90) (100 ,100) (110 ,110) (120 ,120) (130 ,130) (140 ,140) (150 ,150) (160 ,160) (170 ,170) (180 ,180) (190 ,190) (200 ,200) (210 ,210) (220 ,220) (230 ,230) (240 ,240) (250 ,250) (260 ,260) (270 ,270) (280 ,280) (290 ,290) (300 ,300) (310 ,310) (320 ,320) (330 ,330) (340 ,340) (350 ,350) (360 ,360) (370 ,370) (380 ,380) (390 ,390) (400 ,400) (410 ,410) (420 ,420) (430 ,430) (440 ,440) (450 ,450) (460 ,460) (470 ,470) (480 ,480) (490 ,490) (500 ,500) (510 ,510) (520 ,520) (530 ,529) (540 ,539) (550 ,549) (560 ,559) (570 ,569) (580 ,579) (590 ,589) (600 ,599) (610 ,609) (620 ,619) (630 ,629) (640 ,639) (650 ,649) (660 ,659) (670 ,669) (680 ,679) (690 ,689) (700 ,699) (710 ,709) (720 ,719) (730 ,729) (740 ,739) (750 ,749) (760 ,759) (770 ,769) (780 ,779) (790 ,789) (800 ,799) (810 ,809) (820 ,819) (830 ,829) (840 ,839) (850 ,849) (860 ,859) (870 ,869) (880 ,879) (890 ,887) (900 ,897) (910 ,907) (920 ,917) (930 ,927) (940 ,935) (950 ,945) (960 ,955) (970 ,965) (980 ,975) (990 ,985) (1000,995) (1010,1005) (1020,1015) (1030,1025) (1040,1035) (1050,1045) (1060,1055) (1070,1065) (1080,1075) (1090,1085) (1100,1094) (1110,1104) (1120,1114) (1130,1124) (1140,1134) (1150,1144) (1160,1153) (1170,1160) (1180,1170) (1190,1177) (1200,1187) (1210,1197) (1220,1205) (1230,1215) (1240,1225) (1250,1235) (1260,1245) (1270,1255) (1280,1263) (1290,1273) (1300,1283) (1310,1292) (1320,1302) (1330,1312) (1340,1322) (1350,1332) (1360,1342) (1370,1352) (1380,1362) (1390,1372) (1400,1381) (1410,1391) (1420,1399) (1430,1409) (1440,1419) (1450,1428) (1460,1438) (1470,1448) (1480,1458) (1490,1468) (1500,1471) (1510,1479) (1520,1487) (1530,1497) (1540,1507) (1550,1513) (1560,1519) (1570,1524) (1580,1531) (1590,1541) (1600,1551) (1610,1551) (1620,1561) (1630,1568) (1640,1578) (1650,1587) (1660,1597) (1670,1607) (1680,1612) (1690,1620) (1700,1630) (1710,1634) (1720,1642) (1730,1652) (1740,1662) (1750,1672) (1760,1681) (1770,1691) (1780,1701) (1790,1711) (1800,1721) (1810,1731) (1820,1741) (1830,1740) (1840,1742) (1850,1750) (1860,1751) (1870,1758) (1880,1765) (1890,1774) (1900,1779) (1910,1787) (1920,1796) (1930,1806) (1940,1805) (1950,1815) (1960,1819) (1970,1760) (1980,1761) (1990,1771) (2000,1774) (2010,1774) (2020,1778) (2030,1765) (2040,1770) (2050,1776) (2060,1783) (2070,1771) (2080,1781) (2090,1791) (2100,1782) (2110,1786) (2120,1779) (2130,1785) (2140,1777) (2150,1769) (2160,1776) (2170,1782) (2180,1786) (2190,1796) (2200,1804) (2210,1796) (2220,1794) (2230,1804) (2240,1813) (2250,1804) (2260,1792) (2270,1797) (2280,1799) (2290,1804) (2300,1801) (2310,1804) (2320,1810) (2330,1819) (2340,1810) (2350,1776) (2360,1786) (2370,1792) (2380,1780) (2390,1790) (2400,1797) (2410,1806) (2420,1816) (2430,1814) (2440,1820) (2450,1827) (2460,1837) (2470,1819) (2480,1819) (2490,1829) (2500,1839) (2510,1815) (2520,1791) (2530,1800) (2540,1796) (2550,1800) (2560,1802) (2570,1808) (2580,1790) (2590,1791) (2600,1772) (2610,1771) (2620,1762) (2630,1771) (2640,1781) (2650,1783) (2660,1793) (2670,1793) (2680,1798) (2690,1776) (2700,1766) (2710,1774) (2720,1776) (2730,1785) (2740,1780) (2750,1763) (2760,1765) (2770,1772) (2780,1776) (2790,1782) (2800,1789) (2810,1799) (2820,1801) (2830,1807) (2840,1812) (2850,1797) (2860,1807) (2870,1810) (2880,1790) (2890,1798) (2900,1804) (2910,1809) (2920,1815) (2930,1819) (2940,1800) (2950,1805) (2960,1793) (2970,1803) (2980,1795) (2990,1770) (3000,1770) (3010,1767) (3020,1763) (3030,1771) (3040,1775) (3050,1785) (3060,1783) (3070,1784) (3080,1794) (3090,1799) (3100,1797) (3110,1805) (3120,1805) (3130,1804) (3140,1807) (3150,1810) (3160,1802) (3170,1790) (3180,1797) (3190,1807) (3200,1812) (3210,1814) (3220,1821) (3230,1819) (3240,1808) (3250,1806) (3260,1812) (3270,1791) (3280,1793) (3290,1802) (3300,1791) (3310,1779) (3320,1786) (3330,1761) (3340,1764) (3350,1769) (3360,1775) (3370,1781) (3380,1770) (3390,1778) (3400,1780) (3410,1782) (3420,1792) (3430,1777) (3440,1787) (3450,1786) (3460,1791) (3470,1795) (3480,1786) (3490,1769) (3500,1774) (3510,1781) (3520,1788) (3530,1788) (3540,1795) (3550,1801) (3560,1800) (3570,1810) (3580,1805) (3590,1802) (3600,1792) (3610,1802) (3620,1811) (3630,1816) (3640,1825) (3650,1800) (3660,1810) (3670,1800) (3680,1805) (3690,1804) (3700,1813) (3710,1810) (3720,1776) (3730,1785) (3740,1790) (3750,1797) (3760,1807) (3770,1799) (3780,1789) (3790,1764) (3800,1772) (3810,1764) (3820,1773) (3830,1783) (3840,1773) (3850,1765) (3860,1775) (3870,1775) (3880,1783) (3890,1792) (3900,1801) (3910,1803) (3920,1813) (3930,1816) (3940,1804) (3950,1789) (3960,1788) (3970,1794) (3980,1791) (3990,1801) (4000,1805) (4010,1803) (4020,1795) (4030,1801) (4040,1801) (4050,1800) (4060,1796) (4070,1796) (4080,1796) (4090,1785) (4100,1792) (4110,1802) (4120,1812) (4130,1812) (4140,1820) (4150,1826) (4160,1834) (4170,1840) (4180,1838) (4190,1836) (4200,1842)
   };
   \end{axis}
  \end{tikzpicture}
  \caption{Simulation of the stochastic sandpiles model on a $30 \times 30$ box.
  The box starts empty and particles are successively introduced uniformly over
  the square one by one and topple until the configuration stabilizes. Whenever a particle exits
  the box it is eliminated (dissipated). Observe that the overwhelming majority of
  particles get absorbed at the early stages of particle introduction, see
  Theorem~\ref{t:dd} for a precise statement supporting this claim.}
\end{figure}
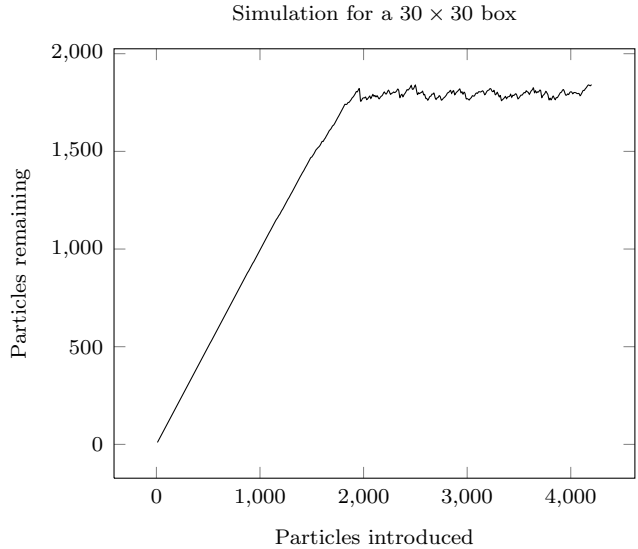

\vspace{4mm}

We now briefly comment on the difficulties encountered in proving such results.
Various proofs of phase transition in statistical mechanics follow what is known as ``energy versus entropy'' approach. The essence of this can be described along the following lines.
In order to show that a certain event does not occur (an event that one doesn't expect at a given phase), one first characterizes a geometric/combinatorial structure that is necessary for its occurrence.
If the number of such structures is overwhelmed by the high energy price to construct them, one can show that such class has vanishing probability.
This is very successful in perturbative statistical mechanics, appearing in different forms such as the Peierls argument or the Pirogov-Sinai theory.
Moreover, the conservation of particles in the system gives rise to long range dependence in time.
This makes it difficult to find a suitable structure behind perpetually continued activity in this model, whose occurrence could be ruled out in this way.

The main techniques that we employ in order to prove our main results are renormalization and sprinkling, that we describe below.

\vspace{4mm}

{\bf Renormalization and sprinkling -} Renormalization techniques provide a way to control how certain parameters of the model vary as we change the scales of the system.
In our case the existence of an absorption phase is intuitive if the initial density of particles is low enough: at time zero, most of regions of space have only few particles which quickly become inactive.
However, fluctuations of the initial density give rise to spatial boxes with an exceptionally high number of particles that will remain active for an extended period of time, or, fluctuations of the dynamics produce boxes, where the initial density is low, but particles are unusually long-term active.
It is expected, though, that such exceptional boxes should not be able to create much damage if we look at them at a larger scale.
Thus, the important parameter to control in our context is the probability $p_k$, that a given box at scale $k$ leaks particles far away in space, see \eqref{e:pk}.
Our multi-scale analysis elucidates the mechanism behind the decay of the probabilities $p_k$ as the scale increases.
Turning this description into a rigorous proof involves controlling the long range dependencies of the model, as well the complications that appear once one conditions on the occurrence of rare events.
We deal with these issues using the Sieving Lemma~\ref{l:couple2}, which allows us to start from an ``almost worst case'' configuration and recover some independence.

The main advantage of this technique is that it looks very little into the details of the dynamics, making us hopeful that they could be adapted to work in other problems.
Moreover, the recursive inequalities that we devise can provide interesting quantitative results, such as Theorems~\ref{t:absorption} and \ref{t:dd}.
The renormalization picture of why absorption happens is very instructive and fits to our intuition that no matter how big a defect is, it will only affect a neighborhood of comparable size.

Another important ingredient that appears in the proof is \emph{sprinkling}. It consists of increasing slightly the density of particles in the system, when needed in order to blur the dependence that may appear by conditioning.
This idea already appeared in \cite{MR2952084} and \cite{PSSS11}.
These techniques become remarkably powerful when combined with the concept of \emph{soft local times} introduced in \cite{PT12}, see the key Lemma~\ref{l:couple2}.

We believe that the methods developed here will help deepen our understanding of absorption transitions and will open various possibilities for future works towards understanding the relation between FES and DDM models under a rigorous perspective.
Moreover, following the general theory of Renormalization Group, one could expect that an a similar multi-scale analysis coukd work not only in perturbative regimes, but lead to results close to criticality, although this would require  several new ingredients.

{\bf Organization -} This paper is organized as follows.
In Section~\ref{s:notation} we establish some notation and define precisely the model of ARW.
Section~\ref{s:triggering} is devoted to showing a triggering statement, which is later bootstrapped into our main quantitative estimates, see Theorem~\ref{t:main} in Section~\ref{s:renorm}.
Our main results are then proved in Section~\ref{s:infinite}.
Sections~\ref{s:simulation}~and~\ref{s:sieving} contain some auxiliary results that are needed through the paper.
Finally we show in Section~\ref{s:other_models} how to extend our results to other FES models such as SSM.

{\bf Acknowledgments -}    V. S. is supported by
Brazilian CNPq grants 308787/2011-0 and 476756/2012-0, and FAPERJ grant
E-26/102.878/2012-BBP and ESF RGLIS grant.
A.T. is grateful for the support from CNPq, through grants 306348/2012-8 and 478577/2012-5.

\section{Notation}
\label{s:notation}

We write $d(x,y)$ for the $\ell_\infty$-distance (also called supremum distance) in $\Z^d$. For $A \subset \mathbb{Z}^d$, we write $d(x,A)$ for $\min_{y \in A} d(x,y)$ and let $\partial_i A$ stand for the internal boundary of $A$, that is $\partial_i A = \{ x \in A; d(x,\mathbb{Z}^d \setminus A) = 1 \}$.

Given $x \in \mathbb{Z}^d$, we let $P_x$ denote the law on $D(\R_+, \Z^d)$ of a continuous time simple random walk starting from $x$. On the space $D(\mathbb{R}_+, \mathbb{Z}^d)$, we denote by $X_t$ the canonical projections, that is for $w \in D(\mathbb{R}_+, \mathbb{Z}^d)$, $X_t(w) = w(t)$.

Through this article we will consider a particle system on $\mathbb{Z}^d$ that is referred to as \emph{activated random walks} or the \emph{sleepy random walkers} in the literature. To make its definition precise, we will describe its generator in detail.

Each site in our system can assume a state in the space $\mathbb{S} = \{0,1,\dots\} \cup \{\mathfrak{s}\}$ where $\mathfrak{s}$ indicates the presence of a sleeping particle there, while $\{0, 1, \dots\}$ stores the number of active particles at that site. Observe from the above that we do not allow more than one sleeping particle at any site.

We denote a configuration of our system by $\eta:\mathbb{Z}^d \to \mathbb{S}$, and given such $\eta$, we define two modifications of this configuration, representing the allowed transformations in our dynamics. The first of them is denoted by $\eta^{(y)}$ (for $y \in \mathbb{Z}^d$) and indicates that the particles at site $y$ attempted to ``sleep''. Note below that sleeping is only allowed if $\eta(y) = 1$.
\begin{equation}
  \eta^{(y)}(x) =
  \begin{cases}
    \eta(x) \qquad & \text{if $x \neq y$ or $\eta(y) \neq 1$ and}\\
    \mathfrak{s} & \text{if $x = y$ and $\eta(y) = 1$.}
  \end{cases}
\end{equation}

The second transformation in our dynamics involves a jump attempt from site $y \in \mathbb{Z}^d$ to some of its neighbors $z$. Observe below that such a jump is only allowed if there is at least one active (that is, not sleeping) particle at $y$. Moreover, once a jump is performed it will wake up a sleepy particle that may be present at $z$. More precisely, if $y$ and $z$ are neighbors
\begin{equation}
  \eta^{(y \to z)}(x) =
  \begin{cases}
    \eta(x) \qquad & \text{if $x \not \in \{y,z\}$ or $\eta(y) \in \{\mathfrak{s},0\}$,}\\
    \eta(y) - 1 & \text{if $x = y$ and $\eta(y) \geq 1$,}\\
    \eta(z) + 1 & \text{if $x = z$, $\eta(y) \geq 1$ and $\eta(z) \neq \mathfrak{s}$ and}\\
    2 & \text{if $x = z$, $\eta(y) \geq 1$ and $\eta(z) = \mathfrak{s}$.}
  \end{cases}
\end{equation}

The transformations above allow us to define the generator of our process as follows. Fixed some sleep rate $\lambda > 0$, for $f: \mathbb{S}^{\mathbb{Z}^d} \to \mathbb{R}_+$, we define
\begin{equation}
  \label{e:generator}
  \mathcal{L}f(\eta) = \sum_{y \in \mathbb{Z}^d} \lambda \big( f(\eta^{(y)}) - f(\eta) \big) + \sum_{y \sim z} \eta(y) \big( f(\eta^{(y \to z)}) - f(\eta)\big).
\end{equation}
Observe that the factor $\eta(y)$ multiplying the last parenthesis above indicates that the rate of jump from $y$ to $z$ is proportional to the number of active particles at $y$. Observe that $\mathfrak{s}$ is only multiplied with zero and we tacitly assume that this product vanishes.

Our particle system is defined on the space $D(\mathbb{R}_+, \mathbb{S}^{\mathbb{Z}^d})$, endowed with the canonical projections $(\eta_t)_{t \geq 0}$ representing the configuration of particles at time $t$. Given an initial collection of particles $(x_j)_{j \in J}$ we can define $\eta_0(y) = |\{j; x_j = y\}|$ to be the number of particles at $y$ at times zero. Supposing that $\eta_0(y)$ does not grow exponentially as $y \to \infty$,
\begin{display}
  \label{e:Pmuxj}
  we define $\mathbb{P}_{(x_j)_{j \in J}}$ to be the law on the space $D(\mathbb{R}_+, \mathbb{S}^{\mathbb{Z}^d})$, starting from $\eta_0$ and under which $\eta_t$ evolves according to the generator $\mathcal{L}$ in \eqref{e:generator}.
\end{display}
To see why such process exists, one could use the theory developed in \cite{MR2108619}, or alternatively use the graphical construction given by the Diaconis-Fulton representation of the process that we introduce in Subsection~\ref{ss:df}. Implicitly, this means that every particle at time zero is active (since the value $\mathfrak{s}$ does not appear in $\eta_0$).

Given a function $\sigma:\mathbb{Z}^d \to \mathbb{R}_+$, we can define a Poisson point process with intensity measure $\sigma$. Most of the times, the intensity $\sigma$ will be taken to be a constant $\zeta > 0$ or $\zeta 1_A$, where $1_A$ stands for the indicator function of $A \subseteq \mathbb{Z}^d$.

Then, we can define $\eta_0(x)$ as the number of points at $x$ in the above Poisson process and
\begin{display}
  \label{e:Pmu}
  define $\mathbb{P}_\sigma$ to be the law on the space $D(\mathbb{R}_+, \mathbb{S}^{\mathbb{Z}^d})$, governing the evolution of $\eta_t$ under $\mathcal{L}$ and starting from a Poisson point process with intensity $\sigma$.
\end{display}
This is well defined whenever $\sigma$ is bounded (which will always be the case throughout this article). We also write, with a slight abuse of notation, $\mathbb{P}_\zeta$ (where $\zeta > 0$) to indicate that the process starts with a Poisson point process with uniform density $\zeta$.

Through this paper, we will be interested in the existence of an absorbing phase for our particle system at low densities. The relevant event will be
\begin{equation}
  \label{e:fixation}
  \{\text{the system fixates}\} = \bigcap_{x \in \mathbb{Z}^d}\{ \text{$\eta(x)$ changes only finitely many times} \}.
\end{equation}
In the literature, this is sometimes referenced as \emph{local fixation}.

With this notation in place, we can state our main results
\newconstant{c:rho_star}
\begin{theorem}
  \label{t:absorption}
  There exists $\useconstant{c:rho_star} > 0$ such that, for every $\zeta < \useconstant{c:rho_star}$
  \begin{equation}
    \mathbb{P}_\zeta [\text{the system fixates locally}] = 1.
  \end{equation}
  which is a consequence of
  \begin{equation}
    \label{e:short_tails}
    \mathbb{P}_\zeta [ \text{number of times that $\eta_t(0)$ changes is larger than $l$}] \leq c \exp\{ - c \log(l)^2\}.
  \end{equation}
\end{theorem}

\newconstant{c:dd}
For the model on a finite box with driving/dissipation, we have
\begin{theorem}
  \label{t:dd}
  Fix any $\zeta < \useconstant{c:rho_star}$ and $\varepsilon > 0$. Then there exists a constant $\useconstant{c:dd} = \useconstant{c:dd}(\varepsilon)$ such that, on a box $B$ of side length $n$
  \begin{equation}
    \label{e:few_exit}
    \mathbb{P}_{\zeta \1{B}} \big[\text{more than $n^{d - 1 + \varepsilon}$ particles exit $B$}\big] \leq c \exp\{ \useconstant{c:dd} \log(l)^2\}.
  \end{equation}
\end{theorem}

Throughout the text we let $c$ denote positive constants, which may depend on $d$ and $\lambda$ and may change from time to time. Further dependence will be denoted explicitly. So, for instance $c(\alpha)$ depends on $\alpha$ and possibly on $d$ and $\lambda$. Moreover, numbered constants such as $c_1, c_2, \dots$ refer to their first appearance in the text.

\subsection{Diaconis-Fulton representation}
\label{ss:df}

This section is dedicated to an important graphical representation of the above described stochastic process.
Such construction was developed in various different articles, such as \cite{MR1218674}, \cite{MR1375419}, \cite{MR2747064} and \cite{MR1622393}.
However, the most complete presentation of this graphical construction for our current purposes has been presented in Rolla's PhD thesis \cite{2008PhDT.......168R}.
The two main advantages of this representation are that it provides notions of monotonicity and commutativity to this system, which are not apparent at first.

We give here a somewhat informal description of the construction, while referring the reader to the more complete description presented in Section~1.3 of \cite{2008PhDT.......168R}.
Intuitively speaking, in this graphical construction, instructions telling what the particles should do (attempt to sleep, jump in such direction...) are stored in each site.
So let us fix a collection $(F_{x,j})_{x \in \mathbb{Z}^d, j \geq 1}$ of elements in the set of instructions  $\Gamma = \{\mathfrak{s}, e_1, -e_1, \dots, e_d, -e_d\}$.
Later we will assign an i.i.d. distribution to the instructions $F_{x,j}$, but for now they could be arbitrary.

We start by describing how the construction goes for a finite initial number $n$ of particles at positions $x_1, \dots, x_n$ having states $\tau_1, \dots, \tau_n \in \{\mathfrak{s}, \mathfrak{a}\}$, where $\mathfrak{s}$ indicates that the particle is sleeping and $\mathfrak{a}$ stands for active.
Whenever a given particle's clock rings, if it is active, it reads the envelope on its current site, executes the corresponding instruction and ``burns that envelope''.
The notion of ``burning'' is implemented through the local time (or total activity) of particles $J \in \{0, 1, \dots\}^{\mathbb{Z}^d}$, that typically starts with value zero everywhere.
The function $J_x$ tells us how many envelopes have been read at site $x$.

To finish our graphical construction, we need a process $T$ giving the time of updates $T = \{0 = t_0 < t_1 < \dots \}$ (later we will take $t_1, t_2, \dots$ to be a Poisson point process) and a sequence of particle indices $N = (n_1, n_2, \dots) \in \{1, \dots, n\}^{\mathbb{Z}_+}$.
The index $n_i$ tells which particle is supposed to act at time $t_i$ (these will later be chosen as i.i.d. and uniformly distributed indices).

We now proceed iteratively as follows
\begin{enumerate}[\qquad 1)]
\item at time $t_i$, the particle $n_i$ reads the envelope $F_{x_{n_i}, J_{x_{n_i}}}$,
\item then, the total activity $J$ is increased by one at position $x_{n_i}$ and
\item the position $x_{n_i}$ and the state $\tau_{n_i}$ of that particle are updated accordingly.
\end{enumerate}

It is not hard to see that the above construction results in the same distribution of particles' evolution (starting at positions $x_1, \dots, x_n$ at states $\tau_1, \dots ,\tau_n$) as long as
\begin{enumerate}[\qquad 1)]
\item $F_{x,j}$ are i.i.d. with proper distributions to reflect the sleeping rate,
\item $T$ is a Poisson Point process with rate $n(1 + \lambda)$ and
\item $N$ is i.i.d. and uniformly distributed over $\{1, \dots, n\}$.
\end{enumerate}

Although the above construction is completely equivalent to the one provided by the generator, it possesses two crucial features which are stated in the theorems below.

We say that a given realization of the graphical construction $\omega = \big( (x, \tau), F, T, N \big)$ \emph{stabilizes} if, by following the above described procedure, we eventually obtain an inactive configuration (that is, $\tau_i = \mathfrak{s}$ for every $i$). This clearly occurs almost surely if $n$ is finite.

\begin{theorem}[Commutativity, Theorem~1.2 in \cite{2008PhDT.......168R}]
  \label{t:commutative}
  Let $\omega = \big( (x, \tau), F, T, N \big)$ be a realization of the graphical construction that stabilizes.
  Suppose that $\omega' = \big( (x', \tau'), F', T', N' \big)$ has the same envelopes and initial state as $\omega$, that is $(x, \tau) = (x', \tau')$ and $F = F'$. Then $\omega'$ also stabilizes and, at the final configuration,
  \begin{equation}
    \text{$\sum_{i=1}^n \delta_{x_i} = \sum_{i=1}^n \delta_{x'_i}$ and $J = J'$}.
  \end{equation}
\end{theorem}

Another very important feature of this construction is related to monotonicity.
Intuitively speaking, if we add more particles or remove $\mathfrak{s}$ envelopes, we sustain activity for a longer time.
We say that $\omega' \preccurlyeq \omega$ if
\begin{enumerate}[\qquad 1)]
\item $\sum\limits_{i = 1}^{n'} \delta_{x'_i} \leq \sum\limits_{i = 1}^n \delta_{x_i}$, that is: $\omega'$ has no more particles on each site than $\omega$,
\item $\sum\limits_{i \leq n';\atop \tau'_i = \mathfrak{a}} \delta_{x'_i} \leq \sum\limits_{i \leq n;\atop \tau_i = \mathfrak{a}}^{\phantom{\infty}} \delta_{x_i}$, that is: $\omega$ has no more active particles on each site than $\omega'$ and
\item the envelopes of $F'$ can be obtained by removing some $\mathfrak{s}$ envelopes from $F$.
\end{enumerate}
We can now state the following
\begin{theorem}[Monotonicity, Theorem~1.3 in \cite{2008PhDT.......168R}]
  \label{t:monotonicity}
  Let $\omega = \big( (x, \tau), F, T, N \big)$ be a realization of the graphical construction that stabilizes.
  Suppose that $\omega' \preccurlyeq \omega$, then $\omega'$ also stabilizes and the final activity of $\omega'$ (number of burned envelopes) is smaller or equal to that for $\omega$, that is $J'_x \leq J_x$ for every $x \in \mathbb{Z}^d$.
\end{theorem}

These theorems will be used several times throughout the text and their proofs can be found in \cite{2008PhDT.......168R}, Section~1.3.
An important consequence of the above stated theorems is the following

\begin{corollary}
  \label{c:offsleep}
  Fix a finite collection of particles $(x_i)_{i \in I}$ and a stopping time $T_i$ for each of them. Then, for any event $A$ which is increasing in the final accumulated activity $(J_x)$,
  \begin{equation}
    \label{e:offsleep}
    \mathbb{P}_{(x_i)_{i \in I}} (A) \leq E_{(x_i)_{i \in I}} \big( \mathbb{P}_{(X^i_{T_i})_{i \in I}} (A) \big),
  \end{equation}
  where the expectation above is taken with respect to an independent collection of simple random walks.
\end{corollary}

Intuitively speaking, the above corollary says that we can turn off the ``sleeping envelopes'', wait until each particle reaches its stopping time and then turn back to the original dynamics.
This procedure can only increase the total activity of the system.

It is also important to note that restricting ourselves to finite starting configurations such as above is not a big restriction due to the following.
Suppose that we are interested in the variable $R$ giving the number of times that $\eta_t(0)$ changed in the whole evolution of the system.
Then $R = \lim_{s \to \infty} R_s$, where $R_s$ is the number of updates in $\eta_t(0)$, for $t \leq s$.
Moreover,
\begin{equation}
  \label{e:finite_approx}
  \mathbb{P}_{\mu}[R_s \geq l] = \lim_{M \to \infty} \mathbb{P}_{\1{B(0,M)} \cdot \mu} [R_s \geq l].
\end{equation}
The proof of this statement can be found in (1.3), Section~1.4 of \cite{2008PhDT.......168R}.

\section{Vanishing density and triggering}
\label{s:triggering}

In this section, we will prove two statements concerning the absorption of particles on a finite box. These statements are important steps towards our final result and their proofs illustrate well of some of the techniques used throughout the paper.

These two results can be viewed as a weaker form of our final Theorem~\ref{t:main}, as they deal with a vanishing density of particles in the box (as the size of the box grows). But they are necessary in order to trigger our renormalization scheme later.

\begin{lemma}
  \label{l:tiny_fraction}
  Let $B_L = [0,L)^d \cap \Z^d$ be a box of side length $L$ and define its half-kernel $B_L' = [L/4, 3L/4)^d \cap \Z^d$. Now fix $\delta > 0$ and an arbitrary collection of particles $(x_j)_{j = 1}^{k}$ in $B_L'$ where $k = \lfloor L^{1-\delta} \rfloor$. Then
  \begin{equation}
    \label{e:tiny_fraction}
    \mathbb{P}_{(x_j)_{j = 1}^{k}} \big[ \text{some particle reaches $\partial_i B_L$} \big] \leq c(\delta) \exp \{ - c L^{\delta/2} \}.
  \end{equation}
\end{lemma}

\begin{proof}
Recall that $d(x,y)$ stands for the supremum norm on $\Z^d$. We now define a sequence of disjoint annuli by
\begin{equation}
  \label{e:A_j}
  A_j = \{x \in \Z^d; d(x,B_L') \in ( j L^{\delta/2}, (j+1) L^{\delta/2}) \}, \text{ for $j = 0, \dots, k $}.
\end{equation}
Note that for $L > c$ all the sets $A_j$ are disjoint and contained inside $B_L \setminus B_L'$, see Figure~\ref{f:traps}.

\begin{figure}[h]
\centering
  \begin{tikzpicture}[scale=.5,font=\tiny]
    \draw[thick] (0,0) rectangle (10,10);
    \draw[thick] (1,1) rectangle (9,9);
    \draw[thick] (2,2) rectangle (8,8);
    \draw[thick] (3,3) rectangle (7,7);
    \draw[gray, very thin] (0.5,0.5) rectangle (9.5,9.5);
    \draw[gray, very thin] (1.5,1.5) rectangle (8.5,8.5);
    \draw[gray, very thin] (2.5,2.5) rectangle (7.5,7.5);
    \draw[fill, white] (9.1,9.1) rectangle (9.9,9.9); \node[below left] at (10.3,10) {$A_2$};
    \draw[fill, white] (8.1,8.1) rectangle (8.9,8.9); \node[below left] at (9.3,9) {$A_1$};
    \draw[fill, white] (7.1,7.1) rectangle (7.9,7.9); \node[below left] at (8.3,8) {$A_0$};
  \end{tikzpicture}
  \caption{An illustration of the annuli $A_j$ defined in \eqref{e:A_j}. As well as the middle rings (in gray) where the particles are stopped.}
  \label{f:traps}
\end{figure}
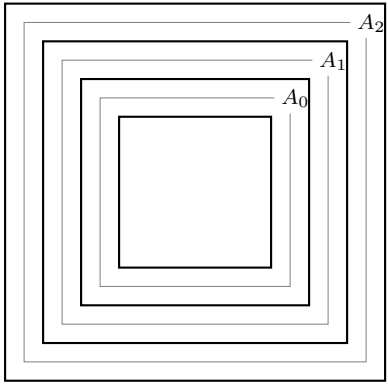

We chose the number of annuli to equal the number of particles so that we can associate them and define the following stopping times
\begin{equation}
  S_j = \inf \{t \geq 0; d(X^j_t, B_L') = \lfloor (j + \tfrac 12) L^{\delta/2} \rfloor \}.
\end{equation}
So that each particle $X^j$ is stopped in the middle of its corresponding annulus $A_j$.

Finally, using Corollary~\ref{c:offsleep},
\begin{equation}
  \begin{split}
    \mathbb{P}_{(x_j)_{j = 1}^{k}} \big[ & \text{some particle
      reaches $\partial_i B_L$} \big]\\
    & \;\; \leq \sup_{\{y_j\}, j \leq k \text{ such that} \atop d(y_j,
      B_L') = \lfloor (j + 1/2) L^{\delta/2} / 2 \rfloor}
    \mathbb{P}_{(x_j)_{j = 1}^{k}}
    \big[\text{some particle $X^j$ reaches $\partial_i A_j$} \big]\\
    & \;\; \leq k \exp\{ - c L^{\delta/2} \},
  \end{split}
\end{equation}
finishing the proof of the lemma by taking $c(\delta)$ large enough.
\end{proof}

In the above lemma, we strongly relied on the fact that the number of particles is much smaller than the side length of the cube, in order to show that they get absorbed before exiting. In what follows we will prove a theorem that deals with a much larger number of particles (now of smaller order than the volume of the cube) but which follows a Poisson distribution.

\begin{theorem}
  \label{t:triggering}
  Fix $\epsilon > 0$. Now let $B_L = [0,L)^d \cap \Z^d$ be a box of side length $L$ and define its kernel $B_L' = [L^\epsilon, L - L^\epsilon)^d \cap \Z^d$. If $\mathbb{P}_L$ denotes the law of the activated random walks process starting from a Poisson point process of particles in $B_L'$ with density $L^{-\epsilon}$ then
  \begin{equation}
    \label{e:triggering}
    \mathbb{P}_L[ \text{some particle reaches $\partial_i B_L$} ] \leq c (\epsilon) \exp \{ - c L^{-\epsilon/4d} \}.
  \end{equation}
\end{theorem}

In the proof of the above theorem we will consider a fine paving of the lattice $\Z^d$ by disjoint boxes, allowing us to use the previous Lemma~\ref{l:tiny_fraction} for each of them. Let us use this opportunity to introduce some definitions which will be useful through the rest of the paper.

We say that a collection of boxes $\{C_i\}_{i \in I}$ is a disjoint paving of $\Z^d$ with side length $\ell$ if there is some $x \in \Z^d$ such that
\begin{equation}
 \label{e:paving}
 \{C_i\}_{i \in I} = \{[0,\ell)^d \cap \Z^d + \ell k + x\}_{k \in \Z^d}.
\end{equation}

For each box $C_i = [0, \ell)^d \cap \Z^d + \ell k$, in the above paving we pick its half-kernel $C_i' = [\ell/4, 3\ell/4) \cap \Z^d + \ell k$ as in Lemma~\ref{l:tiny_fraction} above.


Before going to the proof of Theorem~\ref{t:triggering}, we will analyze what we call the hopping process of a simple random walk and its hitting time of the half-kernels $\cup_{i \in I} C_i'$.

The hopping process will be nothing more than a discrete time Markov chain, obtained when one only looks at the random walk position at some multiples of a fixed constant. Although one could also prove Theorem~\ref{t:triggering} without using the hopping process, we decided to employ it here, since it will be again important in Section~\ref{s:sieving}.

The next lemma deals with a single random walker and will be used to prove Theorem~\ref{t:triggering} and in other parts of the article.

\begin{lemma}
\label{l:fall}
Fix the paving $\{C_i\}_{i \in I}$ and the half-kernels $\{C_i'\}_{i \in I}$ as above. Now fix the arithmetic sequence $t_l = l r$, for $l \geq 0$, where $r \geq \ell^2$ is given. Then the stopping time
\begin{equation}
  \label{e:S_time}
  S = \inf \{ t_l; l \geq 0 \text{ and } X_{t_l} \in \cup_{i \in I} C_i' \}
\end{equation}
is $P_x$-almost surely finite and
\newconstant{c:shake}
\newconstant{c:shake2}
\begin{equation}
  \label{e:not_far}
  \sup_{x \in \Z^d} P_x[\text{$d(X_t,x) \geq D$ for some $t \leq S$}] \leq \useconstant{c:shake} \exp \big\{ - \useconstant{c:shake2} \frac{D}{\sqrt{r}} \big\}.
\end{equation}
\end{lemma}

\begin{proof}
We denote by $\pi_\ell$ the canonical projection from $\Z^d$ to the torus $\mathbb{T}^d_\ell = (\Z/\ell \Z)^d$. Clearly $\pi_\ell$ induces an isomorphism of each half-kernel $C_i'$ into a fixed set $C' \subset \mathbb{T}^d_\ell$.

Using the invariance principle on the torus, since $r \geq \ell^2$ we see that
\begin{equation}
  P_x \big[ X_{t_{l+1}} \in \mcup_{i \in I} C_i' \big| X_{t_l} = y\big] \geq c > 0, \text{ uniformly over $y \in \Z^d$}.
\end{equation}
In view of the above, it becomes clear that $S$ is almost surely finite. Moreover, by the Markov property,
\begin{equation}
  \label{e:S_fast}
  \sup_{x \in \Z^d} P_x[ S \geq s r ] \leq c \exp \{- c s\}.
\end{equation}

We now need to bound the maximal distance traveled by the walker during this time.
For this we use Azuma's inequality for the maximum of a martingale (see for instance \cite{McD98} (41) p.28), and comparing the continuous random walk with the discrete one, we can bound
\begin{align}
  \nonumber
  P_0 \big[ & \max_{t \leq sr} d(X_t, 0) \geq D \big]\\
  & \begin{array}{e}
    & \leq & 2d \exp \Big\{ -\frac{c D^2}{sr} \Big\} + P_0 \big[ \text{$X_t$ does more than $2sr$ jumps before time $sr$} \big]\\[4mm]
    & \leq & c \exp \Big\{ -c \frac{D^2}{sr}\Big\} + c \exp \big\{ - c sr \big\}.
  \end{array}
\end{align}

We now join the above with \eqref{e:S_fast} with $s = D/\sqrt{r}$ to finish the proof of the lemma.
\end{proof}

We are now in position to provide the

\begin{proof}[Proof of Theorem~\ref{t:triggering}]
We first choose a paving $\{ C_i \}_{i \in I}$ of $\Z^d$ with side length $\ell = \lfloor L^{\epsilon/d} \rfloor$. Let $\sum_{j \in J} \delta_{x_j}$ be the realization of a Poisson point process of intensity $L^{-\epsilon} \1{B'_L}$ as in the statement of the theorem. We now define, for $r = \ell^2$, the stopping time $S_j$ as in \eqref{e:S_time} corresponding to the $j$-th particle.

We will use Corollary~\ref{c:offsleep} in order to let the particles walk without sleeping (and consequently not interacting), until the stopping time $S_j$. Note that each particle stops at its own time. At the end of this procedure we hope that the particles will be in such a configuration that allows us to use Theorem~\ref{l:tiny_fraction}

The first issue we have to rule out is that some particle $X_j$ escapes $B_L$ before it stops (at $S_j$). This can be bounded using Lemma~\ref{l:fall} by
\begin{equation}
  \label{e:escape_before_Sj}
  \mathbb{P}_L [\text{some $X_j$ escapes $B_L$ before $S_j$}] \leq \mathbb{P}_L[\# J \geq L^d] + \useconstant{c:shake} L^d \exp \{ - \useconstant{c:shake2} L^{\epsilon - \epsilon/2d}\}.
\end{equation}
The first term is bounded by a simple large deviations estimate and all of the above can be accommodated in the error of \eqref{e:triggering}.

We now claim that
\begin{display}
  \label{e:points_kernels}
  $\omega = \sum_{j \in J} \delta_{X_{S_j}}$ is a Poisson point process, with intensity measure $\nu$ that:\\1) has support on $\cup_{i \in I} C_i'$ and\\2) satisfies $\nu(C_i') \leq \ell^d L^{-\epsilon} \leq 1$.
\end{display}
To see why $\omega$ is a Poisson point process, one simply observe that $\omega$ is a random mapping of another Poisson process, where each particle is mapped independently from each other. To estimate $\nu(C_i')$ we first note that we can stochastically dominate $\omega$ by the process $\omega'$, obtained in the same way as $\omega$, but starting with a Poisson point process of particles on the entire lattice (with intensity $L^{-\epsilon}$). Thus, by ergodicity of $\omega'$, we can infer that its intensity measure satisfies $\nu'(C_i') = \ell^d L^{-\epsilon}$. This together with the domination of $\omega$ by $\omega'$, establishes \eqref{e:points_kernels}.

We now show that with high probability no half-kernel $C'_i$ ends up with more than $L^{\epsilon/2d}$ particles. For this, we use \eqref{e:points_kernels} and a tail estimate on a Poisson($1$) random variable, to obtain that for any given $i \in I$,
\begin{equation}
  \label{e:few_inside}
  \mathbb{E} \Big[ \motimes_{j \in J} P_{x_j} \big[ \# \{j \in J; X^j_{S_j} \in C_i'\} \geq L^{\epsilon/2d} \big] \Big] \leq c \exp \{ - c L^{\epsilon/2d} \},
\end{equation}
where $\mathbb{E}$ stands for the expectation governing $(x_j)_{j \in J}$.

Finally, fix any possible configuration $(y_j)_{j \in J}$ for the $(X^j_{S_j})_{j \in J}$'s such that $\#\{j \in J; y_j \in C_i'\} \leq L^{\epsilon/2d}$ for every $i \in I$. Then we use Lemma~\ref{l:tiny_fraction} to obtain that
  \begin{equation}
    \label{e:tiny_fraction2}
    \mathbb{P}_{(y_j)_{j = 1}^{k}} \big[ \text{some particle reaches $\partial_i B_L$} \big] \leq c(\epsilon) \exp \{ - c L^{\epsilon/4d} \}.
  \end{equation}
This, combined with \eqref{e:escape_before_Sj} and \eqref{e:few_inside} establishes the desired result.
\end{proof}

\section{Renormalization}
\label{s:renorm}

In this section we present the core argument for our main Theorem~\ref{t:main}. For this we will assume the validity of an auxiliary result, the so-called Sieving Lemma, which will be later proved in Section~\ref{s:sieving}.

To establish our main result we follow a multi-scale renormalization argument, which employs the following scale sequence.
Fix $L_0 = 10000$ and define
\begin{equation}
  \label{e:Lk}
  L_{k + 1} = \lfloor L_k^\gamma \rfloor^{2} L_k, \text{ for $k \geq 0$,}
\end{equation}
where $\gamma = 1/10$ from now on.

The values $10000$ and $1/10$ are not intrinsically relevant for the proof and there are other values which would work equally well.

We should think of this scale sequence as growing roughly like $L_{k+1} \sim L_k^{1.2}$.
This can be made precise observing that
\newconstant{c:growth}
\begin{equation}
  \label{e:growth}
  L_{k}^{\gamma} \geq \lfloor L_k^{\gamma} \rfloor \geq \useconstant{c:growth} L_k^{\gamma}, \text{ for every $k \geq 1$}.
\end{equation}
which implies
\begin{equation}
  \label{e:Lk_growth}
  L_k \geq \useconstant{c:growth}^k L_0^{(1+2\gamma)^k} \geq \useconstant{c:growth}^k L_0^{1.2^k}.
\end{equation}

We also introduce the set
\begin{equation}
 \label{e:Mk}
 M_k = \{k\} \times \mathbb{Z}^d, \text{ for $k \geq 0$,}
\end{equation}
which serves to index boxes at scale $k$,
\begin{equation}
 C^0_m = [0, L_k)^d \cap \mathbb{Z}^d + L_k i, \text{ where $m = (k,i) \in M_k$.}
\end{equation}
Note that $\{C^0_m\}_{m \in M_k}$ form a disjoint tiling of $\mathbb{Z}^d$ with side length $L_k$.

We will also need the intermediate scale
\begin{equation}
 \label{e:Rk}
 R_{k+1} = \lfloor L_k^\gamma \rfloor L_k \; \; \big(= L_{k+1}/\lfloor L_k^\gamma \rfloor \big), \text{ for $k \geq 0$,}
\end{equation}
which lies between $L_k$ and $L_{k+1}$.
\newconstant{c:Rsmall}
Note that for $k \geq \useconstant{c:Rsmall}$, we have that $5 R_k \leq L_k$, so that we can define the (non-empty) boxes
\begin{equation}
  \label{e:C12m}
  C^1_m = [R_k, L_k - R_k)^d \cap \Z^d + L_k \cdot i \quad \text{ and } \quad C^2_m = [2R_k, L_k - 2R_k)^d \cap \Z^d + L_k \cdot i,
\end{equation}
where again $m = (k,i) \in \mathbb{Z}^d$.
Note that $C^2_m \subset C^1_m \subset C^0_m$ for every $m \in M_k$, see Figure~\ref{f:boxes}.

The boxed $C^1$ and $C^2$ are kernels (in a sense similar to that introduced in Theorem~\ref{t:triggering}). The annuli between $C^2 \setminus C^1$ and $C^1 \setminus C^0$ can be thought as buffer zones that start without particles.

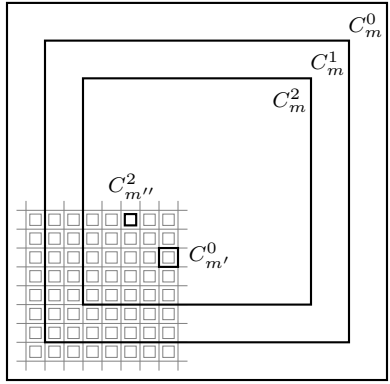
\begin{figure}[h]
\centering
  \begin{tikzpicture}[scale=.5,font=\tiny]
    \foreach \x in {1,...,9} { \draw[color=gray,very thin] (0.5*\x,0.25) -- (0.5*\x,4.75); }
    \foreach \y in {1,...,9} { \draw[color=gray,very thin] (0.25,0.5*\y) -- (4.75,0.5*\y); }
    \foreach \x in {1,...,8} { \foreach \y in {1,...,8} { \draw[color=gray,very thin] (0.5*\x + .1, 0.5*\y + .1) rectangle (0.5*\x + .4, 0.5*\y + .4); } }
    \draw[thick] (0,0) rectangle (10,10); \node[below left] at (10.2,10) {$C^0_m$};
    \draw[thick] (1,1) rectangle (9,9); \node[below left] at (9.2,9) {$C^1_m$};
    \draw[thick] (2,2) rectangle (8,8); \node[below left] at (8.2,8) {$C^2_m$};
    \draw[thick] (4,3) rectangle (4.5,3.5); \node[right] at (4.5,3.3) {$C^0_{m'}$};
    \draw[thick] (3.1,4.1) rectangle (3.4,4.4); \node[above] at (3.3,4.5) {$C^2_{m''}$};
  \end{tikzpicture}
  \caption{An illustration of the boxes $C^q_m$ ($q = 0, 1, 2$) and sub-boxes at the previous scale.}
  \label{f:boxes}
\end{figure}

We can now define the main quantity we want to bound
\begin{equation}
  \label{e:pk}
  \begin{split}
    p_k(\zeta) & = \sup_{m \in M_k} \mathbb{P}_{\zeta \1{C_m^2}}\big[ \text{some particle reaches $\partial_i C^0_m$} \big]\\
    & = \mathbb{P}_{\zeta \1{C^2_{(k,0)}}}\big[ \text{some particle reaches $\partial_i C^0_{(k,0)}$} \big].
  \end{split}
\end{equation}
Recall the definition of $\mathbb{P}_\sigma$ in \eqref{e:Pmu}.

Using the above definition, we can obtain the following consequence of Theorem~\ref{t:triggering}
\begin{corollary}
\label{c:trigger}
For $k \geq c$, we have
\begin{equation}
  \label{e:cor_trigger}
  p_k(L_k^{-1/2}) \leq c \exp\{-c L_k^{\frac{1}{8d}}\}.
\end{equation}
\end{corollary}
The above corollary will be used to trigger the recursion relation on the $p_k$'s in Theorem~\ref{t:main} below.

\begin{proof}
Note first that
\begin{equation}
  R_k = \frac{L_k}{\lfloor L_{k-1}^\gamma \rfloor} = \lfloor L_{k-1}^\gamma \rfloor L_{k-1} \overset{\eqref{e:growth}}\geq \useconstant{c:growth} L_{k-1}^{1 + \gamma} \geq \useconstant{c:growth} L_k^{\frac{1 + \gamma}{1 + 2\gamma}} = \useconstant{c:growth} L_k^{11/12},
\end{equation}
which grows much faster than $L^{1/2}$. Therefore, we choose $\epsilon = 1/2$ in Theorem~\ref{t:triggering}.

Observe that $C_{(k,0)}^0 = B_{L_k}$ and for $k \geq c$, we have $R_k \leq \sqrt{L_k}$, so that $C_{(k,0)}^2 \subseteq B'_{L_k}$ (see definition in Theorem~\ref{t:triggering}). Then, by Theorem~\ref{t:monotonicity},
\begin{equation}
  \begin{split}
    p_k(L_0^{-1/2}) & = \mathbb{P}_{L_0^{-1/2} \1{C_{(0,0)}^2}}
    [\text{some particle reaches $\partial_i C_{(0,0)}^0$}]\\
    & \overset{}\leq \mathbb{P}_{L_0^{-1/2} \1{B'_{L_k}}}
    [\text{some particle reaches $\partial_i C_{(k,0)}^0$}] \leq c \exp\{-c L_0^{\frac{1}{8d}}\},
  \end{split}
\end{equation}
finishing the proof of the Corollary.
\end{proof}

The next result improves on the one above, since it works for non-vanishing densities.

\medskip

Throughout our proof, we will employ an argument called ``sprinkling'', in which a slight change in the density of the system compensates for the interdependence between boxes. For this reason, we introduce the sequence of densities
\begin{equation}
  \label{e:rhok}
  \zeta_k = \zeta_0 \big(1 - \frac{1}{4}\sum_{j=1}^k \frac{1}{j^2}\big), \text{ for any given $\zeta_0 > 0$.}
\end{equation}
The exact choice of $\zeta_k$'s is not crucial for our arguments to work. What we need from this sequence is that $\inf_k \zeta_k \geq \zeta_0/2$ and that $\zeta_k - \zeta_{k-1}$ is sufficiently large, so that we can apply a large deviations estimate in the proof of the Sieving Lemma~\ref{l:couple2} in Section~\ref{s:sieving}.

The following result is the key ingredient to prove our main theorem. It relates the probabilities $p_{k+1}$ with the probabilities in the previous scale $p_k$.

\newconstant{c:main1}
\newconstant{c:main3}
\newconstant{c:main4}
\begin{theorem}
\label{t:main_aux}
There exists $\useconstant{c:main1}, \useconstant{c:main3}, \useconstant{c:main4}$ such that, if $k \geq \useconstant{c:main1}$ and $\zeta_0 \in \big(8(k+1)^2 L_k^{-\gamma/3}, 1\big)$, we have
\begin{equation}
  \label{e:main_aux}
  p_{k+1}(\zeta_{k+1}) \leq \Big( \frac{L_{k+1}}{L_k} \Big)^{2d} p_k(\zeta_k)^2 + \useconstant{c:main3} \exp\{ - \useconstant{c:main4} L_k^{\gamma/3}\}, \text{ for $k \geq 0$}.
\end{equation}
\end{theorem}

The main steps in the proof of the above theorem are schematically summarized in Figure~\ref{f:bounds}. Roughly speaking, each of the boxes in that picture represent a starting configuration under which we are evaluating the probability that some particle escapes.

In Figure~\ref{f:bounds}, regions are either: gray, hatched or filled with small squares.
These patterns have different meanings that we now explain. Gray regions represent a configuration starting from a Poisson point process (like in the definition of $p_k(\zeta)$). We now explain the hatched area, which stands for \emph{balanced} configurations.

\begin{definition}
  \label{d:balanced}
  We say that a collection $(x_j)_{j \in J}^n$ is $\zeta$-balanced with respect to the paving $\{C_i\}_{i \in I}$ of side length $L$ if for every $i \in I$, the number of points falling in $C_i$ is smaller or equal to $\zeta L^d$. Recall the definition of paving in \eqref{e:paving}.
\end{definition}
The arrow between (a) and (b) in Figure~\ref{f:bounds} simply indicates that with high probability a Poisson point process is balanced.

We now need to explain what the areas filled with small gray squares represent in Figure~\ref{f:bounds}. These are the so-called \emph{sieved} configurations, where all particles are positioned in squares of type $C^2_{m'}$, with $m'$ in scale $k-1$. Having the particles in such situation, will allow us to relate the probability that some escape in terms of $p_{k-1}$, which is the main part of our induction argument. We now define precisely what these configurations are.

\begin{definition}
  \label{d:sieved}
  Given a density $\zeta$, an index $q \in \{1,2\}$ and $m \in M_k$ ($k \geq \useconstant{c:Rsmall}$), we define the $q$-sieved law in $C_m$ with density $\zeta$ as the law of a Poisson point process with intensity
\begin{equation}
  \mu^q_m(\zeta) := \zeta \;\cdot \; \1{C^q_m} \; \cdot \; \1{\cup_{m' \in M_{k-1}} C^2_{m'}} \; .
\end{equation}
The support of the above measure is illustrated in Figure~\ref{f:bounds} (c) and (e) for $q = 1$ and $2$ respectively (gray area).

\end{definition}

\begin{figure}[h]
\centering \includegraphics[width = \textwidth]{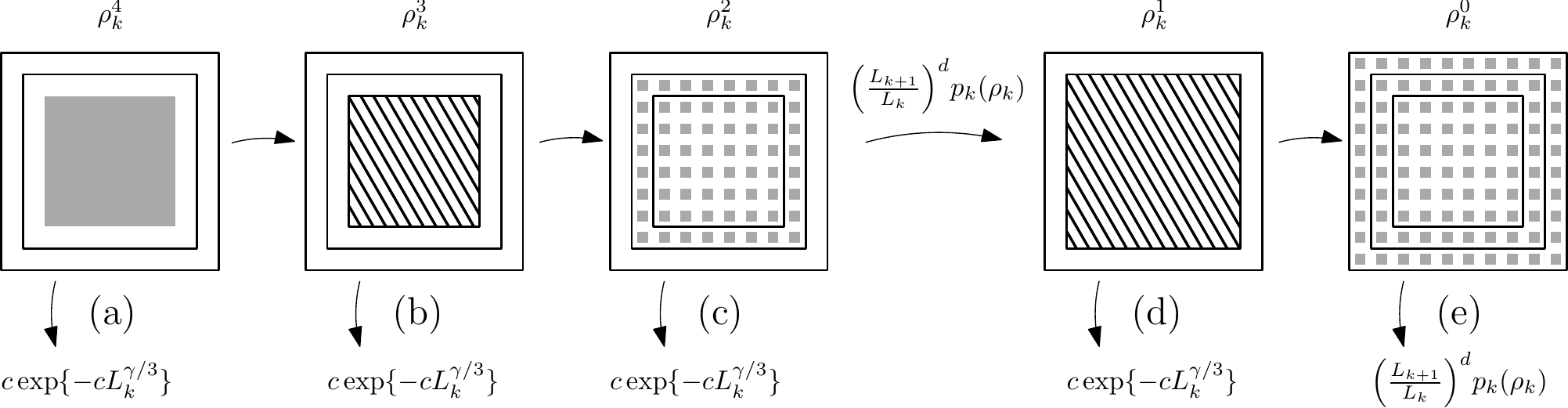}
  \caption{In this figure we represent informally the steps in the proof of Theorem~\ref{t:main_aux}. To emphasize this correspondence, we label the equations in the proof of Theorem~\ref{t:main_aux} according to what they correspond in this figure. The arrows (b $\rightarrow$ c) and (d $\rightarrow$ e) illustrate the use of the Sieving Lemma~\ref{l:couple2}.}
  \label{f:bounds}
\end{figure}

Another important aspect of our proof is that in every step depicted in Figure~\ref{f:bounds}, we need to perform a sprinkling of our density $\zeta_k$. This motivates us to introduce the intermediate densities between $\zeta_{k+1}$ and $\zeta_k$:
\begin{equation}
  \label{e:rhokr}
  \zeta_k^r = \zeta_k - \frac{r \zeta_0}{16(k+1)^2}, \text{ for $r = 0, \dots, 4$,}
\end{equation}
so that $\zeta_{k+1} = \zeta_{k}^4 < \zeta_{k}^3 < \dots < \zeta_{k}^0 = \zeta_{k}$.

\medskip

We are now ready to state our main auxiliary result in the proof of Theorem~\ref{t:main_aux}, the so-called \emph{Sieving Lemma}. Roughly speaking, in this lemma we start from a balanced configuration of particles, let them run until a certain stopping time, ending up with a configuration which is dominated by a sieved law as in Definition~\ref{d:sieved}. For this to work, we have to allow for some sprinkling in the intensity of particles ($r+1 \to r$) and for an enlargement in the size of the box ($C^{q+1}_m \to C^{q}_m$). This is all made precise below.

For the following, fix either $(r,q) = (0,1)$ or $(r,q) = (2,2)$.
\newconstant{c:couple2}
\begin{lemma}[\bf{Sieving Lemma}]
\label{l:couple2}
Fix $k > \useconstant{c:couple2}$, an index $m \in M_{k+1}$ and a collection $(x_j)_{j \in J} \subset C_m^{q+1}$ which is $(\zeta^{r+1}_k)$-balanced with respect to the paving $\{C_{m'}\}_{m' \in M_k}$. Then, assuming that
\begin{equation}
\label{e:rho_large}
1 \geq \zeta_0 \geq 8 (k+1)^2 L_k^{-\sfrac{\gamma}{3}},
\end{equation}
we can find a coupling $\mathbb{Q}$ between $\bigotimes_{j \in J} P_{x_j}$ and the sieved law with intensity $\mu^q_m(\zeta^r_k)$, in such a way that
\begin{equation}
    \mathbb{Q} \Big[ \sum_{j \in J} \delta_{X^j_{T_j}} \leq \sum_{j'
      \in J'} \delta_{Y_{j'}}, \text{ and no $X^j$ hits $\partial
      C_m^{q}$ before $T_j$}
    \Big] \geq 1 - c \exp \{ -c L_k^{\sfrac{\gamma}{3}} \},
\end{equation}
where $T_j$ is a stopping time for the $j$-th particle (defined in \eqref{e:T_j}).
\end{lemma}

The above lemma is illustrated in Figure~\ref{f:bounds}, where the cases $(r,q) = (0,1)$ or $(2,2)$ correspond respectively to the arrows (d $\rightarrow$ e) and (b $\rightarrow$ c).

\begin{proof}[Proof of Theorem~\ref{t:main_aux}]
In order to use Lemma~\ref{l:couple2}, recall that we assumed \eqref{e:rho_large} in Theorem~\ref{t:main_aux}. We start estimating $p_{k+1}(\zeta_{k+1})$ (recall \eqref{e:pk}) using Theorem~\ref{t:monotonicity}. For this, fix any given $m \in M_{k+1}$ and write
\begin{align}
  \nonumber p_{k+1}(\zeta_{k+1}) & = \mathbb{P}_{\zeta_k^4 \1{C_{m}^2}} \big[\text{some particle reaches $\partial_i C_{m}^0$} \big]\\
  \tag{a $\to$ b}
  \label{e:term_a}
  & \leq \mathbb{P}_{\zeta_k^4 \1{C_{m}^2}} \big[\text{$(X_0^j)_{j \leq J}$ is not $\zeta_k^3$-balanced with respect to $(C_{m'}^0)_{m' \in M_k}$} \big]\\
  \nonumber & \quad + \sup_{(x_j)_{j \in J} \subset C^2_{m} \atop \text{$\zeta_k^3$-balanced}} \mathbb{P}_{(x_j)_{j \in J}} \big[\text{some particle reaches $\partial_i C_{m}^0$} \big].
\end{align}
The above terms correspond to the arrows leaving the box (a) in Figure~\ref{f:bounds}.

We start by estimating the first term above. By a crude large deviation's estimate,
\begin{equation}
  \tag{a}
  \label{e:a_b}
  \begin{array}{e}
    \mathbb{P}_{\zeta_k^4 \1{C_{(k,0)}^2}} \big[& (X_0^j)_{j \leq J} & \text{ is not $\zeta_k^3$-balanced with respect to $(C_{m'}^0)_{m' \in M_k}$} \big]\\
    & \leq & \Big(\frac{L_{k+1}}{L_k}\Big)^d \mathbb{P}_{\zeta^4_k} \big[ \text{there are more than $\zeta^3_kL^d_k$ particles in $C^0_{(k,0)}$} \big]\\
    & \leq & \Big(\frac{L_{k+1}}{L_k}\Big)^d \exp \{ \zeta^4_k L_k^d (e^\theta -1) - \theta \zeta^3_k L_k^d \}\\
    & \overset{\theta \leq 1}\leq & \Big(\frac{L_{k+1}}{L_k}\Big)^d \exp \{ - \theta(\zeta^3_k - \zeta^4_k) L_k^d + \theta^2 \zeta_0 L_k^d \}\\
    & \overset{\eqref{e:rhokr}}\leq & \Big(\frac{L_{k+1}}{L_k}\Big)^d \exp \Big\{ - \theta \zeta_0 L_k^d \Big( \frac{1}{16(k+1)^2} - \theta \Big) \Big\}\\
    & \overset{\theta = 1/32(k+1)^{2}}\leq & \Big(\frac{L_{k+1}}{L_k}\Big)^d \exp \Big\{ - \frac{\zeta_0}{[32(k+1)^2]^2} L_k^d \Big\}\\
    & \overset{\eqref{e:rho_large}}\leq & \Big(\frac{L_{k+1}}{L_k}\Big)^d \exp \{ -\tfrac 1{128(k+1)^2} L_k^{d-\gamma/3} \} \overset{k>c}\leq c \exp \{- c L_k^{\gamma/3}\}.
  \end{array}
\end{equation}
The above bound corresponds to the arrow going down from (a) in Figure~\ref{f:bounds}.

We now turn to the second term in \eqref{e:term_a}, using the Diaconis-Fulton representation, together with \eqref{e:rho_large} and Lemma~\ref{l:couple2} we have
\begin{equation}
  \tag{b $\to$ c}
  \label{e:term_b}
  \begin{split}
    \sup_{(x_j)_{j \in J} \subseteq C^2_m \atop \text{$\zeta_k^3$-balanced}} \mathbb{P}_{(x_j)_{j \in J}} \big[ & \text{some particle reaches $\partial_i C_{m}^0$} \big]\\
    & \leq c \exp \{ - c L_k^{\gamma/3} \} + \mathbb{P}_{\mu^1_m(\zeta^2_k)} \big[ \text{some particle reaches $\partial_i C_{m}^0$} \big].
  \end{split}
\end{equation}
The above terms correspond to the two arrows leaving (b) in the figure.

We now turn to the bound on the second term above. Every particle in the starting configuration of $\mathbb{P}_{\mu^1_m(\zeta^2_k)}$ belongs to $C^2_{m'} \subset C^1_m$ for some $m'\in M_{k-1}$, see Definition~\ref{d:sieved}. Consider the following stopping time,
\begin{equation}
  \label{e:D_1}
  D_1 = \inf \big\{ t \geq 0; \text{some particle reaches $\cup_{m' \in M_k} \partial_i C^0_{m'}$} \big\}.
\end{equation}
Clearly if $D_1 = \infty$, all particles have slept before reaching the boundary of the larger box $\partial_i C^0_m$. Then, by the Markov property
\begin{equation}
  \label{e:term_c}
  \begin{split}
    \mathbb{P}_{\mu^1_{(k,0)}(\zeta^2_k)} \big[ & \text{some particle reaches $\partial_i C_{(k,0)}^0$} \big]\\
    & \leq \mathbb{E}_{\mu^1_{(k,0)}(\zeta^2_k)} \Big[  D_1 < \infty,  \mathbb{P}_{(X^j_{D_1})} \big[ \text{some particle reaches $\partial_i C_{(k,0)}^0$} \big] \Big]\\
    & \leq \mathbb{P}_{\mu^1_{(k,0)}(\zeta^2_k)} [D_1 < \infty] \sup_{(x^j)_{j \in J} \subseteq C^1_m \atop \text{$\zeta^1_k$-balanced}} \mathbb{P}_{(x^j)_{j \in J}} \big[ \text{some particle reaches $\partial_i C_{(k,0)}^0$} \big]\\
    & \quad + \mathbb{P}_{\mu^1_{(k,0)}(\zeta^2_k)} [(X^j_0)_{j \in J} \text{ is not $\zeta^3_k$-balanced}],
  \end{split}
\end{equation}
where in the last inequality, we used the fact that no particle can exit its corresponding box $C^0_{m'}$ ($m' \in M_k$) before time $D_1$.

Observe now that the same calculation as in \eqref{e:a_b} can be used to bound the last probability in \eqref{e:term_c}, so that the above can be bounded by
\begin{equation}
  \label{e:term_c2}
  \tag{c $\to$ d}
  \begin{split}
    \Big( \frac{L_{k+1}}{L_k} & \Big)^d p_k(\zeta_k) \sup_{(x^j)_{j \in J} \subseteq C^1_m \atop \text{$\zeta^1_k$-balanced}} \mathbb{P}_{(x^j)_{j \in J}} \big[ \text{some particle reaches $\partial_i C_{(k,0)}^0$} \big] + c \exp \{ - c L_k^{\gamma/3} \}\\
    \overset{\text{Lemma~\ref{l:couple2}}}\leq & \Big( \frac{L_{k+1}}{L_k} \Big)^d p_k(\zeta_k) \Big( \mathbb{P}_{\mu^0_{(k,0)}(p^0_k)} \big[ \text{some particle reaches $\partial_i C_{(k,0)}^0$} \big] + c \exp \{- c L_k^{\gamma/3} \} \Big)\\
    & \quad + c \exp \{ - c L_k^{\gamma/3} \}
  \end{split}
\end{equation}
and similarly to \eqref{e:term_c} we can bound $\mathbb{P}_{\mu^0_{(k,0)}(p^0_k)} \big[ \text{some particle reaches $\partial_i C_{(k,0)}^0$} \big]$ by simply $\mathbb{P}_{\mu^0_{(k,0)}(\zeta^0_k)} [D_1 < \infty] \leq \big(\frac{L_{k+1}}{L_k}\big)^d p_k$. This, together with \eqref{e:term_a}, \eqref{e:term_b} and \eqref{e:term_c2} gives
\begin{equation}
  \begin{split}
    p_{k+1}(\zeta_{k+1}) \leq & c \exp\{-c L_k^{d-\gamma/3} \} + c \exp \{ - c L_k^{\gamma/3}\}\\
    & + \Big( \frac{L_{k+1}}{L_k} \Big)^d p_k(\zeta_{k}) \Big( \Big( \frac{L_{k+1}}{L_k} \Big)^d p_k(\zeta_k) + c \exp \{ - c L_k^{\gamma/3}\} \Big),
  \end{split}
\end{equation}
finishing the proof of Theorem~\ref{t:main_aux}.
\end{proof}

We now can use the above recursion to prove the decay of $p_k$ in the following

\newconstant{c:density}
\newconstant{c:t_main}
\begin{theorem}
\label{t:main}
There exists $\useconstant{c:density}, \useconstant{c:t_main} > 0$ such that, for $\zeta < \useconstant{c:density}$,
\begin{equation}
  p_k(\zeta) \leq \useconstant{c:t_main} \exp \{ - \log^2 L_k \}.
\end{equation}
\end{theorem}

\begin{proof}
\newconstant{c:kzero}
We will use Theorem~\ref{t:main_aux} to prove the desired result by induction. For that let us first pick $\useconstant{c:kzero}$ so that for any $k \geq \useconstant{c:kzero}$
\begin{equation}
  \label{e:kzero}
  L_k^{4d \gamma} \exp\{ - (2 - 1 - 2\gamma) \log^2 L_k \} + \useconstant{c:main3} \exp \{ - \useconstant{c:main4}  L_k^{\gamma/3} \} \leq 1,
\end{equation}
which can clearly be done. Recall $\useconstant{c:main3}$ and $\useconstant{c:main4}$ from Theorem~\ref{t:main_aux}.

Now, recalling Corollary~\ref{c:trigger} we can choose $\bar k \geq \useconstant{c:kzero}$ such that for any $k \geq \bar k$ one has
\begin{gather}
  \label{e:first_pk}
  p_{k}(L_{k}^{-1/2}) \leq \exp \big\{- \log^2 L_{k} \big\} \text{ and}\\
  \label{e:k_lower}
  L_{k}^{-1/2} \geq 8(k+1)^2 L_{k}^{-\gamma/3}.
\end{gather}
We now fix $\useconstant{c:density} = L_{\bar k}^{-1/2}$ and prove that for $\zeta_0 = \useconstant{c:density}$ we have
\begin{equation}
  \label{e:log2ind}
  p_{k}(\zeta_k) \leq \exp \big\{- \log^2 L_{k} \big\}, \text{ for $k \geq \bar k$},
\end{equation}
by induction in $k$.

We use \eqref{e:first_pk} and the definition of $\zeta_0$ to establish \eqref{e:log2ind} for $k = \bar k$. Now, supposing that we have established \eqref{e:log2ind} for some $k \geq \bar k$, we recall \eqref{e:k_lower} and Theorem~\ref{t:main_aux} to obtain
\begin{equation*}
\begin{split}
  \frac{p_{k+1}(\zeta_{k+1})}{\exp \{ - \log^2 L_k \}} & \leq L_k^{4d \gamma} \exp\{ - 2 \log^2 L_k + \log^2 L_{k+1} \} + \useconstant{c:main3} \exp \{ - \useconstant{c:main4} L_k^{\gamma/3} + \log^2 L_{k+1}\}\\
  & \leq L_k^{4d \gamma} \exp\{ - (2 - 1 - 2\gamma) \log^2 L_k \} + \useconstant{c:main3} \exp \{ - \useconstant{c:main4} L_k^{\gamma/3} \} \overset{\eqref{e:kzero}}\leq 1,
\end{split}
\end{equation*}
so that $p_{k+1} \leq \exp \{ - \log^2 L_{k+1}\}$. Finishing the proof of \eqref{e:log2ind} by induction, which yields Theorem~\ref{t:main} by properly choosing $\useconstant{c:t_main}$.
\end{proof}

In the next section, we will employ the above results to show the existence of an absorbed phase for the infinite system.

\begin{remark}
\label{r:pho_large}
The condition \eqref{e:rho_large} is essential in proving the Sieving Lemma~\ref{l:couple2} (see Section~\ref{s:sieving}). This is justifiable, since a small value of $\zeta_0$ will make the sprinkling parameter also smaller and consequently breaking the good decay that appears in Lemma~\ref{l:couple}. This restriction is the main reason for the existence of Section~\ref{s:triggering}, where we deal with particle densities that vanish, but which are still large enough to satisfy \eqref{e:rho_large}.
\end{remark}

\section{The absorption phase}
\label{s:infinite}

We start this section with an application of Theorem~\ref{t:main}, that estimates the maximum displacement of particles when we start from sets other than a box. This will be later used in the proof of Theorem~\ref{t:absorption}.

\begin{theorem}
\label{t:set_displace}
For any given $A \subset \mathbb{Z}^d$ and $\zeta \leq \useconstant{c:density}/(4 \cdot 3^d)$, we have
\begin{equation}
  \mathbb{P}_{\zeta \1{A}} \big[ \text{some particle exits $B(A,s)$} \big] \leq c |A| \exp \{ - c \log^2 s \}.
\end{equation}
\end{theorem}

The main idea behind this proof is to sieve the particles into boxes $C^2_{m'}$ that are contained in $B(A,s)$ and then use Theorem~\ref{t:main}. The main obstacle to employ this strategy will be to adapt the Sieving Lemma~\ref{l:couple2} to deal with the situation where the particles do not start inside a kernel (either $C^1_m$ or $C^2_m$).

\begin{figure}[h]
\centering \includegraphics[width = .6\textwidth]{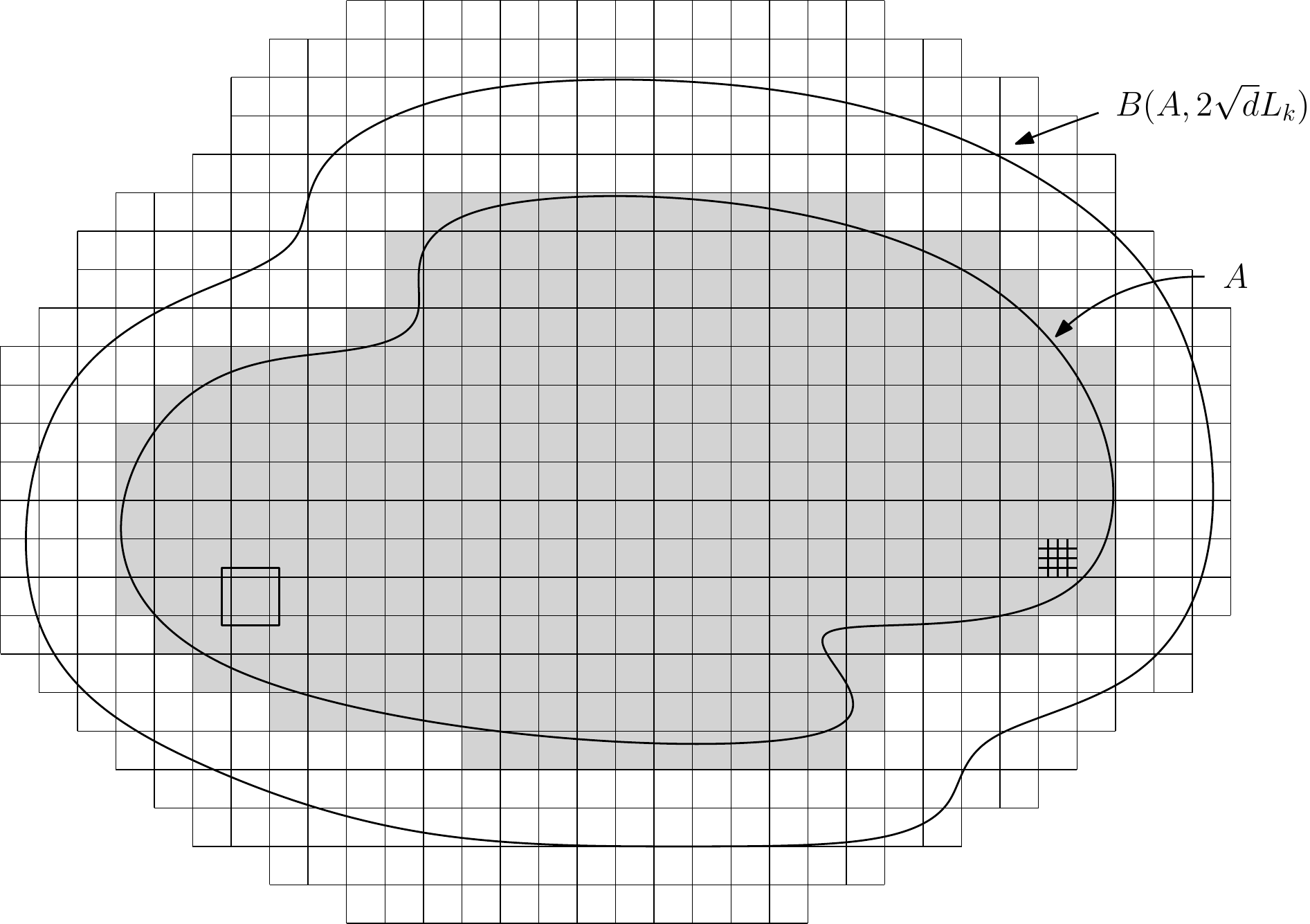}
  \caption{The sets $A$ and $B(A, 2 \sqrt{d} L_k$). We have also drawn the paving $\{C^1_i\}_{i \in I}$ and shown one box $C^0_i$ in the left and a few boxes $\{C^0_{m'}\}_{m' \in M_{k-1}}$ in the right.}
  \label{f:set_a}
\end{figure}

\begin{proof}
Given $s$ as in the statement of the theorem, we choose $k = k(s)$ such that
\begin{equation}
  L_{k-1} < \frac{s}{2\sqrt{d}} \leq L_k,
\end{equation}
recall \eqref{e:Lk}.

Observe that we can assume $\zeta = \useconstant{c:density}/(4 \cdot 3^d)$ and choose $\zeta_0 = 2\zeta$ and $\zeta_k$ as in \eqref{e:rhok}. Note that with this choice
\begin{equation}
  \label{e:rho_bounded}
  \zeta \leq \zeta_k \leq 2\zeta, \text{ for every $k \geq 1$.}
\end{equation}

Pick $c$ large enough so that, if $s \geq c$, we have
\begin{align}
  \label{e:s_large}
  & L_k - 2R_k \geq L_k/2 \text{ and}\\
  & L_k^{\gamma/3} \geq 8(k+1)^2/\zeta,
\end{align}
as in \eqref{e:rho_large}. Therefore, we are allowed to use the Sieving Lemma~\ref{l:couple2}.

Consider a paving $\{C^1_i\}_{i \in I}$ of $\mathbb{Z}^d$ by boxes of side length $L_k - 2 R_k$ (which coincides with that of $C^1_m$ in \eqref{e:C12m}), see Figure~\ref{f:set_a}. Note that this paving is not made of $C^0_m$ boxes as in the previous section. In particular, the $C^0_m$ boxes corresponding to the $C^1_m$ above overlap.

This paving can be subdivided into the finer (scale $k - 1$) paving $\{C^0_m\}_{m \in M_{k-1}}$ as indicated (for one box) in Figure~\ref{f:set_a}.

Our first requirement is for the initial Poisson point process of particles to be balanced with respect to this finer paving. More precisely, we observe that $A$ can touch at most $|A|$ boxes in $\{C^0_m\}_{m \in M_{k-1}}$, and arguing as in \eqref{e:a_b} we obtain
\begin{equation}
  \begin{array}{e}
    \mathbb{P}_{\zeta \1{A}} \big[ (X^j_0)_{j \in J} \text{ is not $(2 \zeta)$-balanced} \big]  & \leq & |A| \exp \{ \zeta L_{k-1}^d (e^\theta -1) - 2 \theta \zeta L_{k-1}^d \}\\
    & \overset{\theta = 1/2 \atop \eqref{e:s_large}} \leq & |A| \exp \{ - c \zeta L_{k-1}^d\}
  \end{array}
\end{equation}
where the term \emph{balanced} above refers of course to the paving $\{C^0_m\}_{m \in M_{k-1}}$. Thus, we get by the Diaconis-Fulton representation that
\begin{equation}
  \label{e:balance_displace}
  \begin{split}
    \mathbb{P}_{\zeta \1{A}} \big[ & \text{some particle exits $B(A,s)$} \big]\\
    & \leq \sup_{(x^j)_{j \in J} \subseteq A \atop \text{$2\zeta$-balanced}} \mathbb{P}_{(x^j)_{j \in J}} \big[ \text{some particle exits $B(A,s)$} \big] + |A| \exp \{ - c \zeta L_{k-1}^d\}.
  \end{split}
\end{equation}

We are preparing to employ the Sieving Lemma~\ref{l:couple2} for the balanced collection above. For this, given a box in the coarse paving $C^1_i$, consider the corresponding box $C^0_i$ having the same center and side length $L_k$. We again observe that the boxes $C^0_i$ are not disjoint as their corresponding $\{C^1_i\}$. However,
\begin{equation}
  \label{e:3d}
  \text{each box $C^0_i$ intersects at most $3^d$ other such boxes,}
\end{equation}
see detail in Figure~\ref{f:set_a}.

We now restrict our attention to one box $C^1_i$ intersecting $A$. If the points $(x_j)$ are $2\zeta$-balanced (with respect to $\{C^0_m\}_{m \in M_{k-1}}$) as in \eqref{e:balance_displace}, we can split them into two collections that are $\zeta^1_k$-balanced, so that we can use the Sieving Lemma~\ref{l:couple2} and \eqref{e:rho_bounded} twice, to obtain that
\begin{equation}
  \begin{split}
    \sup_{(x^j)_{j \in J} \subseteq C^1_i \atop \text{$2\zeta$-balanced}} \mathbb{P}_{(x^j)_{j \in J}} & \big[ \text{some particle exits $B(A,s)$} \big]\\
    & \leq \mathbb{P}_{\mu^i(4\zeta)} \big[ \text{some particle exits $B(A,s)$} \big] + c \exp \{ - c L_{k-1}^{\gamma/3}\},
  \end{split}
\end{equation}
where $\mu^i(4\zeta)$ stands for the measure $4\zeta \cdot \1{C^0_i} \cdot \1{\cup_{m' \in M_{k-1}}C^2_{m'}}$.

Taking the union over all possible choices of $C^1_i$ intersecting $A$, we obtain
\begin{equation}
  \label{e:sieve_displace}
  \begin{split}
    \sup_{(x^j)_{j \in J} \subseteq A \atop \text{$2\zeta$-balanced}} & \mathbb{P}_{(x^j)_{j \in J}} \big[ \text{some particle exits $B(A,s)$} \big]\\
    & \overset{\eqref{e:3d}}\leq \mathbb{P}_{\mu(4 \cdot 3^d \zeta)} \big[ \text{some particle exits $B(A,s)$} \big] + c |A| \exp \{ - c L_{k-1}^{\gamma/3}\},
  \end{split}
\end{equation}
with $\mu(4 \cdot 3^d \zeta) = 4 \cdot 3^d \zeta \cdot \1{}\{\cup_i \text{$C^0_i$ intersecting $A$}\} \cdot \1{}\{\cup_{m' \in M_{k-1}} C^2_{m'}\}$.

To finish the proof, we use the following observation. Start the particles with law given by the intensity measure $\mu(4 \cdot 3^d \zeta)$, if every particle that starts at a given $C^2_{m'}$ does not leave the corresponding $C^0_{m'}$, then they did not exit $B(\tilde{A},\sqrt{d} L_k)$, where $\tilde{A} = \cup_i C^0_i$, where the union is taken over $C^0_i$ intersecting $A$. Since $\tilde{A} \subseteq B(A, \sqrt{d}L_k)$, we conclude that no particle had left $B(A,2\sqrt{d}L_k) \subset B(A,s)$ as desired.

We can therefore join \eqref{e:balance_displace}, \eqref{e:sieve_displace} and Theorem~\ref{t:main} to obtain that
\begin{equation}
  \begin{split}
    \mathbb{P}_{\zeta \1{A}} & \big[ \text{some particle} \text{ exits $B(A,s)$} \big]\\
    & \leq |A| \exp \{ -c \zeta L_{k-1}^{\gamma/3} \} + c |A| \exp \{ -c \log^2(L_{k-1}) \} \leq c |A| \exp \{ - c \log^2 s \}.
  \end{split}
\end{equation}
This finishes the proof of Theorem~\ref{t:set_displace} (since $4 \cdot 3^d \zeta = \useconstant{c:density}$). We may need to properly choose the constants, in order to take care of the case $s < c$.
\end{proof}

We are now in position to prove the existence of an absorption phase for our model. But before, let us quickly present the

\begin{proof}[Proof of Theorem~\ref{t:dd}]
Given a box $B$ of side length $n$, without loss of generality we may suppose that $B = [0,n)^d \cap \mathbb{Z}^d$. We define $B'$ to be $[n^{\varepsilon/2}, n - n^{\varepsilon/2})^d \cap \mathbb{Z}^d$.

Using the Diaconis-Fulton representation, it easy to see that
\begin{equation*}
  \begin{split}
    \mathbb{P}_{\zeta \1{B}} \big[\text{more than $n^{d - 1 + \varepsilon}$} & \text{ particles leave $B$}\big]\\
    & \leq \mathbb{P}_{\zeta \1{B}} \big[\text{more than $n^{d - 1 + \varepsilon}$ particles start in $B \setminus B'$}\big]\\
    & \quad + \mathbb{P}_{\zeta \1{B'}} \big[ \text{some particles exits $B$} \big].
  \end{split}
\end{equation*}
For $n$ large enough, the volume of $B \setminus B'$ is smaller or equal to $(1/2) n^{d - 1 + \varepsilon}$. Therefore, the first term above is bounded by $c \exp\{-c n^{d - 1 + \varepsilon}\}$, recall that $\zeta \leq 1$.

The bound on the second term of the above equation is obtained by a simple application of Theorem~\ref{t:set_displace}. This finishes the proof of Theorem~\ref{t:dd}.
\end{proof}

We now turn to the proof of our main result. The proof has two steps. First we ``turn off sleeping'' and let the particles walk until a stopping time (that depends on the particle), in order to accommodate them into a nested collection of annuli surrounding the origin, see the right hand side of Figure~\ref{f:S_in_out}. Then we ``turn sleeping back on'' and let the particles run, hopping that different annuli will not interact. For this we use Theorem~\ref{t:set_displace}.

\begin{proof}[Proof of Theorem~\ref{t:absorption}]
By \eqref{e:finite_approx} and Borel-Cantelli's lemma, it is clearly enough to show \eqref{e:short_tails}. Given $l$ as in \eqref{e:short_tails}, we choose $i_0$ such that
\begin{equation}
  \label{e:i_0}
  2^{10d (i_0 + 1) } \leq l < 2^{10d (i_0 + 2)},
\end{equation}
which is greater or equal to $4$ if $l \geq c$.

Fixed such value $i_0 \geq 4$ we partition the space $\mathbb{Z}^d$ into the disjoint annuli
\begin{equation}
  \label{e:annuli}
  B_i = \big\{ x \in \mathbb{Z}^d; 2^{i - 1} < |x| \leq 2^{i} \big\}, \text{ for $i \geq i_0 + 1$},
\end{equation}
together with the ball $B_{i_0} = B(0,2^{i_0})$.

We start by decomposing the desired probability as follows
\begin{equation}
  \label{e:eta_changes}
  \begin{split}
    \mathbb{P}_\zeta \big[\#\{ & \text{times $\eta(0)$ changes} \} > l\big]\\
    & \leq \mathbb{P}_{\zeta} \big[ \{ \text{some particle visits more than three $B_i$'s} \} \big]\\
    & \quad + \mathbb{P}_{\zeta} \big[\#\{ \text{visits to the origin when particles are killed on $B_{i_0 + 3}$}\} > l]
  \end{split}
\end{equation}
Let us start bounding the second term in the right hand side. We do this by first bounding the probability that more than $2\zeta \cdot \text{vol}(B(0,2^{i_0+2}))\}]$ particles start in $B_{i_0} \cup B_{i_0 + 1} \cup B_{i_0 + 2}$. Plus the probability that some of them takes longer than $2^{3(i_0 + 2)}$ to hit $B_{i_0 + 3}$. Observe that any particle starting at $B_{i_0} \cup B_{i_0 + 1} \cup B_{i_0 + 2}$ has a positive probability of hitting $B_{i_0 + 3}$ after $2^{2(i_0 + 2)}$ steps. Moreover, this probability is bounded away from zero, uniformly over the starting point and independently of $i_0$. Therefore,
\begin{equation}
  \label{e:i_0+3}
  \begin{split}
    \mathbb{P}_{\zeta} & \big[\#\{ \text{visits to the origin when particles are killed on $B_{i_0 + 3}$}\} > l]\\
    & \overset{\eqref{e:i_0}}\leq \mathbb{P}_\zeta [\#\{\text{particles in $B_{i_0} \cup B_{i_0 + 1} \cup B_{i_0 + 2}$}\} > 2\zeta \cdot \text{vol}(B(0,2^{i_0+2}))\}]\\
    & \qquad + 2\zeta \cdot \text{vol}(B(0,2^{i_0+2}) \sup_{x \in B_{i_0} \cup B_{i_0 + 1} \cup B_{i_0 + 2}}P_x[H_{B_{i_0 + 3}} > 2^{3(i_0 + 2))}]\\
    & \leq \exp \{ -c 2^{i_0} \} + c \zeta 2^{d i_0}\exp \{ - c 2^{i_0} \} \overset{\eqref{e:i_0}}\leq c \exp \{-c l^{10^{-d}}\}.
  \end{split}
\end{equation}

In order to bound the first term in the right-hand-side of \eqref{e:eta_changes} we need to introduce some further notation.
Given $i \geq i_0 + 1$ and any point $x \in B_i$, we associate to $x$ the two spheres in $\mathbb{Z}^d$
\begin{equation}
  \begin{split}
    & S_{\text{out}}(x) = \Big\{ y \in \mathbb{Z}^d; \lfloor |y| \rfloor = \lfloor |x| \rfloor + 3 \cdot 2^{i-2} \Big\} \text{ and},\\
    & S_{\text{in}}(x) = \Big\{ y \in \mathbb{Z}^d; \lfloor |y| \rfloor = \lfloor |x|/4 \rfloor + 3 \cdot 2^{i-4} \Big\},
  \end{split}
\end{equation}
see Figure~\ref{f:S_in_out} (left). Finally, if $x \in B_{i_0}$, then we set $S_{\text{in}}(x) = \varnothing$ and
\begin{equation}
  S_{\text{out}}(x) = \Big\{y; \lfloor |y| \rfloor = \lfloor |x|/2 \rfloor + 5 \cdot 2^{i_0 - 2} \Big\}.
\end{equation}

The exact form of the above definition is not important. The only properties of these sets that we need are
\begin{gather}
  \label{e:S_size}
  |S_{\text{in}}(x)| \text{ and } |S_{\text{out}}(x)| \leq c |x|^{d-1}, \text{ for $x \in B_i$, $i \geq i_0+1$},\\
  S_{\text{in}}(x) \subset B_{i-1} \text{ and } S_{\text{out}}(x) \subset B_{i+1}, \text{ for $x \in B_i, i \geq i_0 + 1$},\\
  d(S_{\text{in}}(x), \partial_i B_{i-1}) \text{ and } d(S_{\text{out}}(x), \partial_i B_{i+1}) \geq c|x|, \text{ for $x \in \mathbb{Z}^d$ and}\\
  \label{e:S_overlap}
  \# \big\{y \in B_i; S_{\text{in}}(y) = S_{\text{in}}(x) \text{ or } S_{\text{out}}(y) = S_{\text{out}}(x) \big\} \leq c |x|^{d-1}, \;\; x \in \mathbb{Z}^d,
\end{gather}
which are easy to verify from their definitions.

\begin{figure}[h]
\centering \includegraphics[width = .6\textwidth]{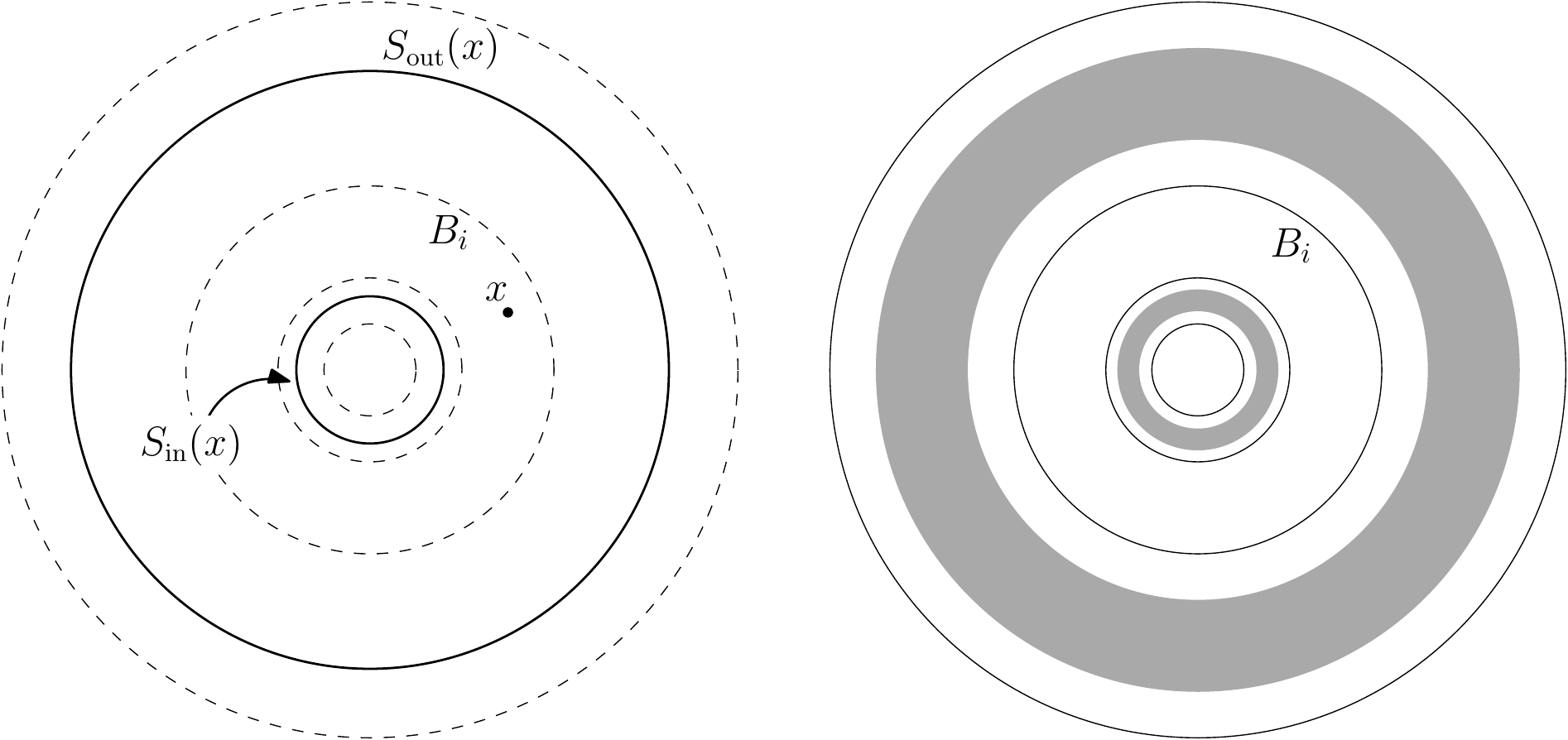}
  \caption{In the left, we depict a point $x \in B_i$ and its corresponding $S_{\text{in}}$ and $S_{\text{out}}$ (the dashed lines represent the boundaries of $B_i$'s). In the right, the gray area represents the union of $S_{\text{in}}(x)$ and $S_{\text{out}}(x)$, for all $x \in B_i$.}
  \label{f:S_in_out}
\end{figure}

We now define the stopping time
\begin{equation}
  T = \inf \{t \geq 0; X_t \in S_{\text{in}}(X_0) \cup S_{\text{out}}(X_0) \},
\end{equation}
which clearly satisfies $T < \infty$, $P_x$-a.s. We also define the set
\begin{equation}
  D = \mcup_{x \in \mathbb{Z}^d} S_{\text{in}}(x) \cup S_{\text{out}}(x),
\end{equation}
see Figure~\ref{f:S_in_out} (right), depicting two annuli from $D$.

For $i \geq i_0$ and $x \in B_i$, we use Lemma~7.6 in \cite{PT12}, to conclude that
\begin{equation}
  \label{e:uniform_entrance}
  P_x[X_T = y] \leq c |y|^{d-1}, \text{ for any $y \in S_{\text{in}}(x) \cup S_{\text{out}}(x)$.}
\end{equation}

Joining \eqref{e:S_size}, \eqref{e:S_overlap} and \eqref{e:uniform_entrance}, we obtain that: if $(X^j_0)_{j \geq 1}$ is distributed as a Poisson point process on $\mathbb{Z}^d$ with density $\zeta$, then
\newconstant{c:rho_D}
\begin{display}
  \label{e:pit_stop}
  under $\bigotimes_{j \geq 1} P_{X^j_0}$, the distribution of $(X^j_T)_{j \geq 1}$ follows a Poisson point process with intensity bounded by $\useconstant{c:rho_D} \zeta \cdot \1 D$.
\end{display}

We can now bound the first term in the right-hand-side of \eqref{e:eta_changes}. For this, fix any $\zeta \leq \useconstant{c:density}/(\useconstant{c:rho_D} 4 \cdot 3^d)$, and note that
\begin{equation*}
  \begin{split}
    \mathbb{P}_{\zeta} \big[ \{ & \text{some particle visits more than three $B_i$'s} \} \big]\\
    & \overset{\eqref{e:pit_stop}}\leq \mathbb{P}_{\useconstant{c:rho_D} \zeta \1 D} \big[ \{ \text{some particle leaves its corresponding $B_i$} \} \big]\\
    & \leq \sum_{i \geq i_0} \mathbb{P}_{\useconstant{c:rho_D} \zeta \1{D \cap B_i}} \big[ \text{some particle leaves $B_i$} \big]\\
  \end{split}
\end{equation*}
which, by Theorem~\ref{t:set_displace} is smaller or equal to
\begin{equation}
  \sum_{i \geq i_0} c (2^i)^d \exp\{ - c \log^2(c 2^i)\} \leq c \int_{c \log(x)}^\infty \exp\{ - s^2 \} \d s \leq \frac{1}{c \log(x)} \exp \{ - c \log(x)^2\},
\end{equation}
where $x = 2^{i_0}$. This, together with \eqref{e:i_0}, \eqref{e:eta_changes} and \eqref{e:i_0+3} finishes the proof of the theorem.
\end{proof}

\section{Simulation and domination of particles}
\label{s:simulation}

The main result of this section is the Lemma~\ref{l:couple} which is the main ingredient to prove the Sieving Lemma~\ref{l:couple2} in Section~\ref{s:sieving}. The Sieving Lemma was used several times during the proof of our main results. This section is devoted to the construction of a coupling between a collection of independent random walks and a Poisson point process.

\begin{remark}
  The Lemma~\ref{l:couple} below is very similar in spirit to Proposition~4.1 in \cite{PSSS11}. However, the results in \cite{PSSS11} are not formulated in the most convenient way for our use here (most notably, it concerns Brownian motions instead of random walks). Therefore, for the reader's convenience, we present a self contained proof of this coupling here. We would also like to point out that our proof follows a different approach than that of \cite{PSSS11}, employing the concept of \emph{Soft local times} from \cite{PT12} instead.
\end{remark}

Let us collect some simple facts concerning the heat kernel of a continuous time, simple random walk on $\Z^d$. Let $p_t(x,y)$ stands for the heat kernel, given by $p_t(x,y) = P_x[X_t = y]$. When confusion may arise, we write $p^d_t$ instead of $p_t$, to explicit the dimension under consideration.

\begin{lemma} \textnormal{($d=1$)}
 \label{l:monotoneheat}
 For any $t \geq 0$, the heat kernel $p_t(0,x)$ is monotone non-increasing in $x \geq 0$.
\end{lemma}

\begin{proof}
 We should use reversibility and couple two walks starting at different positions.
\end{proof}

\begin{lemma} \textnormal{($d \geq 1$)}
 \label{l:localclt}
 For any $t \geq 0$ and $x = (x_1, \dots, x_d)$,
 \begin{align}
  \label{e:indcoord}
  p^d_t(0,x) & = \prod_{k = 1}^d p^1_t(0,x_k) \text{ and}\\
  \label{e:localclt}
  p^1_t(0,x) & \leq \frac{c}{\sqrt{t}}, \text{ for every $t \geq 0$, $x \in \Z$}.
 \end{align}
\end{lemma}

\begin{proof}
 The first claim \eqref{e:indcoord} is a consequence of the independence of the evolution of coordinates of a continuous time random walk, while the second follows from standard random walk theory.
\end{proof}

The above lemma will be used in the proof of Lemma~\ref{l:integration}, which deals with the integration of the heat kernel under a balanced  cloud of starting particles. More precisely, Lemma~\ref{l:integration} is crucial in the proof of Lemma~\ref{l:couple}.

\medskip

We know that the heat kernel $p_t$ sums up to one on its second coordinate. What the lemma below attempts to do is to approximate this feature when we are not summing all the possible points in the second coordinate, but over a balanced collection $(x_j)_{j \in J}$, recall definition~\ref{d:balanced}. More precisely,

\begin{lemma}\textnormal{($d \geq 1$)}
 \label{l:integration}
 Let $\{C_i\}_{i \in I}$ be a paving of $\Z^d$ by disjoint boxes of side length $L$ in the sense of \eqref{e:paving}, see ~\ref{f:heat}. Consider also an $\zeta$-balanced collection $(x_j)_{j \in J}$ with respect to $\{C_i\}_{i \in I}$.
 Then, for any $t \geq 0$,
 \begin{equation}
  \label{e:integration}
  \sum_{j \in J} p_t(0,x_j) \leq \zeta \Big( 1 + \frac{cL}{\sqrt{t}} + \dots + \frac{c L^d}{\sqrt{t}^d} \Big).
 \end{equation}
\end{lemma}

\begin{figure}[h]
\centering \includegraphics[width = 0.7\textwidth]{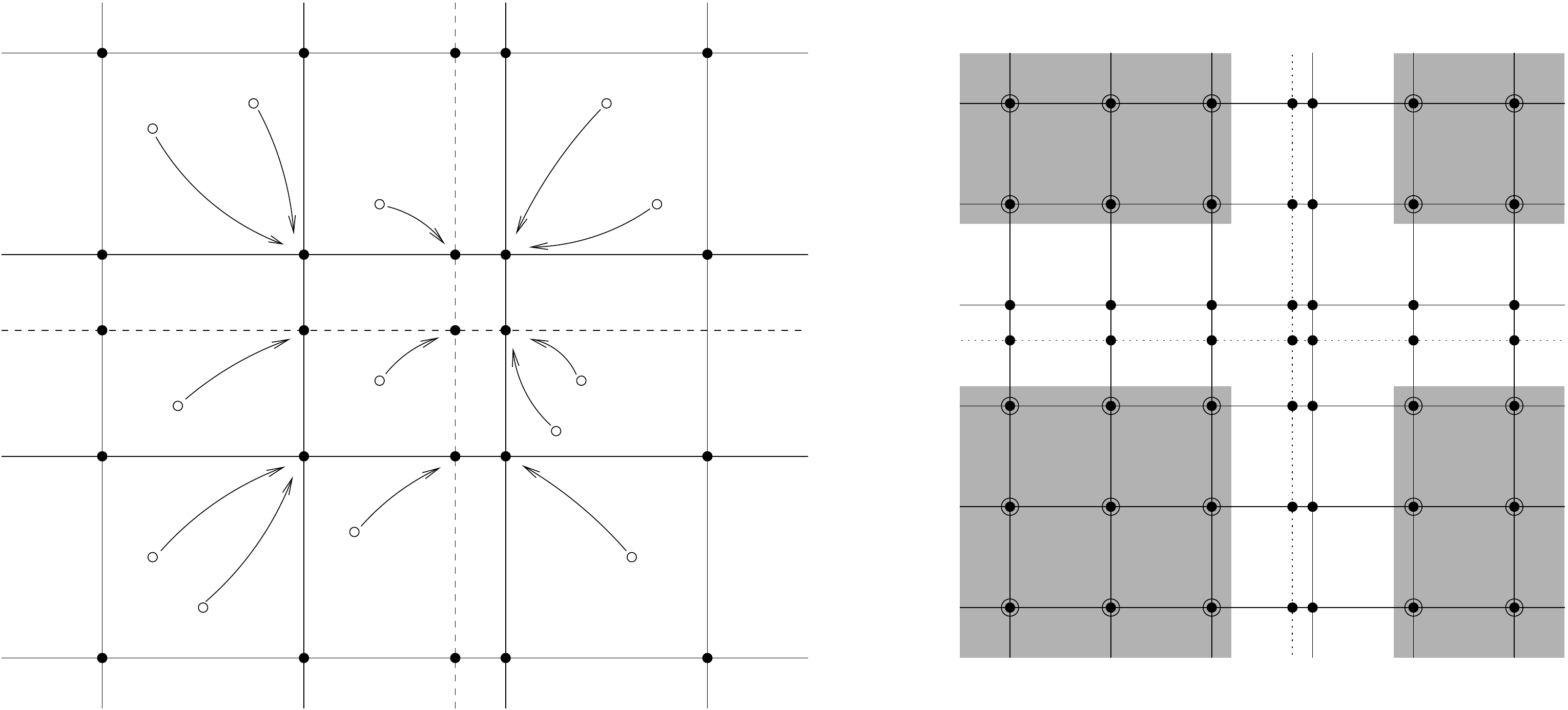}
  \caption{The boxes $\{C_i\}_{i \in I}$ and the points $x_j$ (circles) and their corresponding $\phi(x_j)$ (filled). Note the coordinate axes represented by dashed lines.}
  \label{f:heat}
\end{figure}

\begin{proof}
 We introduce the set $E = \{x \in \Z^d; \text{ at least one coordinate of $x$ is zero}\}$ composed of coordinate hyper-planes. Also, define the set of box indices
 \begin{equation}
  I' = \{i \in I; C_i \text{ does not intersect } E\}.
 \end{equation}
 Since the boxes $C_i$ are connected, for every $i \in I'$,
 \begin{display}
  \label{e:samesign}
  all points $x = (x_1, \dots, x_d) \in C_i$ have the same signs $s_i = \big(\tfrac{x_1}{|x_1|}, \dots, \tfrac{x_d}{|x_d|} \big)$.
 \end{display}
 Therefore, we can define the map $\phi: I' \to I$ that sends an index $i$ to the index of the box $C_{\phi(i)} = C_i + s_i L$. The importance of this map lies on the fact that
 \begin{align}
  \label{e:phiinject}
  & \phi \text{ is injective and}\\
  \label{e:phismall}
  & \min_{x \in C_i} p_t(0,x) \geq \max_{x \in C_{\phi(i)}} p_t(0,x), \text{ for every $t \geq 0$ and $i \in I'$.}
 \end{align}
 Indeed, this is a simple consequence of Lemma~\ref{l:monotoneheat} and the fact that the absolute value of any given coordinate of a point $x$ in $C_i$ is smaller than that of a point $x' \in C_{\phi(i)}$.

 We now split the balanced set $(x_j)_{j \in J}$ into the indices $J' = \{j \in J; x_j \in \bigcup_{i \in I'} C_{\phi(i)}\}$ and $J'' = J \setminus J'$. First, let us estimate, for any $t \geq 0$,
 \begin{equation}
  \label{e:Jprime}
  \begin{array}{e}
  \sum_{j \in J'} p_t(0,x_j) & \overset{\eqref{e:phiinject}}{\leq} & \zeta L^d \sum_{i \in I'} \max_{x \in C_{\phi(i)}} p_t(0,x) \overset{\eqref{e:phismall}}\leq \zeta L^d \sum_{i \in I'} \min_{x \in C_i} p_t(0,x)\\
  & \leq & \zeta \sum_{i \in I'} \sum_{x \in C_i} p_t(0, x) \leq \zeta.
  \end{array}
 \end{equation}

 Given a point $x = (x_1, \dots, x_d) \in J''$, we first claim that
 \begin{equation}
  \label{e:closetoaxis}
  \text{there exists a coordinate $k_x \in \{1, \dots, d\}$ for which $|x_k| \leq 3L$.}
 \end{equation}
 Supposing that the above does not hold, the signs $s = (\tfrac{x_1}{|x_1|}, \dots, \tfrac{x_d}{|x_d|})$ will be well defined and let $i \in I$ be the index of the box containing $y = x - Ls$. We will prove that $i$ belongs to $I'$ and $x \in C_{\phi(i)}$, contradicting the fact that $x \in J''$. For this, observe that all the coordinates of $y$ have absolute value at least $2L$, in which case, any other point $z \in C_i$ has its coordinates with absolute value larger or equal to $L$, implying that $i \in I'$. The fact that $x \in C_{\phi(i)}$ follows from the definition of $\phi$ below \eqref{e:samesign}. This contradiction establishes \eqref{e:closetoaxis}. For $x \in J''$, we define the projection $\tilde x = (x_1, x_2, \dots, 0, \dots, x_{d-1}, x_d)$, where we nullify the $k_x$-th coordinate.

 We will now show \eqref{e:integration} by induction in $d$. For $d = 1$, the result is obvious from \eqref{e:closetoaxis}, since there are at most $6 \zeta L$ points in $J''$. Now supposing that \eqref{e:integration} holds for $d - 1$ and given an $\zeta$-balanced collection $(x_j)_{j \in J} \subset \Z^d$,
 \begin{equation}
  \sum_{j \in J} p_t(0,x_j) \overset{\eqref{e:Jprime}}\leq \zeta + \sum_{j \in J''} p_t(0,x_j) \overset{\text{Lemma~\ref{l:monotoneheat}}}\leq \zeta + \sum_{j \in J''} p_t(0, \tilde x_j).
 \end{equation}
 Observe that for each $j \in J''$, the point $\tilde x_j$ belongs to a coordinate hyperplane, we therefore partition $J''$ into $J''_1, \dots, J''_d$, corresponding to each of these hyperplanes. Moreover, $(\tilde x_j)_{j \in J''_k}$ is a $6 \zeta L^{d-1}$-balanced collection in $\Z^{d-1}$. Thus, we can use our induction hypothesis to conclude that
 \begin{equation}
  \begin{array}{e}
  \sum_{j \in J} p^d_t(0,x_j) & \overset{{\eqref{e:indcoord} \atop \eqref{e:localclt}}}\leq & \zeta + \frac{c}{\sqrt{t}} \sum_{k = 1}^d \sum_{j \in J''_k} p^{d-1}_t(0,\tilde x_j)\\
  & \leq & \zeta + \frac{c\zeta L}{\sqrt{t}} \Big( 1 + \frac{cL}{\sqrt{t}} + \dots + \frac{c L^{d-1}}{\sqrt{t}^{d-1}} \Big)\\
  & \leq & \zeta \Big( 1 + \frac{cL}{\sqrt{t}} + \dots + \frac{c L^{d}}{\sqrt{t}^{d}} \Big),
  \end{array}
 \end{equation}
 finishing the proof of Lemma~\ref{l:integration}.
\end{proof}

Given a sequence $(x_j)_{j \in J}$ of points in $\mathbb{Z}^d$, we denote by $\bigotimes_{j \in J} P_{x_j}$ the law of an independent sequence of random walks starting on each of these points. Still in this context, we denote by $X^j_t$ the canonical coordinates of the $j$-th particle at time $t$.

Then next lemma provides us with a way to couple a collection of independent particles  with a Poisson point process. This procedure is very much inspired by the works in \cite{PSSS11} and \cite{PT12}.

\newconstant{c:couple}
\begin{lemma}
 \label{l:couple}
 Let $(x_j)_{j \in J} \subset \mathbb{Z}^d$ be an $\zeta$-balanced collection with respect to the paving $\{C_i\}_{i \in I}$, of side length $L$. Then, given $\zeta > 0$, we can then find a coupling $\mathbb{Q}$ between $\bigotimes_{j \in J} P_{x_j}$ and the law of a Poisson point process $\sum_{j' \in J'} \delta_{Z_j}$ on $\mathbb{Z}^d$ with intensity $\zeta' \geq \zeta$, in such a way that
 \begin{equation}
  \mathbb{Q} \Big[ \1{D} \cdot \sum_{j \in J} \delta_{X^j_t} \leq \1{D} \cdot \sum_{j' \in J'} \delta_{Z_{j'}} \Big] \geq 1 - |D| \; \exp \Big\{ - (\zeta' - \zeta)L^d + c \zeta \frac{L^{d+1}}{\sqrt{t}} \Big\},
 \end{equation}
 for any set $D \subset \Z^d$ and every $t \geq \useconstant{c:couple} L^2$.
\end{lemma}

\begin{proof}
 Using Corollary~\ref{c:couplesystem}, we obtain a coupling $\mathbb{Q}$ such that
 \begin{equation}
  \label{e:coupleGG}
  \mathbb{Q} \Big[ \1{D} \cdot \sum_{j \in J} \delta_{X^j_t} \leq \1{D} \cdot \sum_{j' \in J'} \delta_{Z_j} \Big] \geq \mathbb{Q} \big[ G_J(z) \leq \zeta', \text{ for every $z \in D$}\big],
 \end{equation}
 where $G_J(z) = \sum_{j \in J} \xi_j \; p_{t}(x_j, z)$ and $\xi_j$ is an i.i.d. sequence of Exp($1$) distributed random variables.

 We are going to estimate the right hand side of \eqref{e:coupleGG} using concentration inequalities. For this, let us first estimate the expectation
 \begin{equation}
  \label{e:EGJ}
  E^{\mathbb{Q}}[G_J(z)] = \sum_{j \in J} p_{t}(0, x_j - z) \; \; \overset{\mathclap{\text{Lemma~\ref{l:integration}}}}\leq \quad \; \zeta \Big( 1 + \dots + \frac{cL^d}{\sqrt{t}^d} \Big) \overset{t \geq cL^2}\leq \zeta \Big( 1 + c\frac{L}{\sqrt{t}} \Big).
 \end{equation}
 Given $z \in D$, we write $p_j$ for $p_{t}(x_j, z)$ and estimate
 \begin{align}
  \nonumber
  \mathbb{Q} \big[ & G_J(z) > \zeta', \text{ for some $z \in D$}\big] \leq |D| \cdot \mathbb{Q} \Big[ G_J(z) > \zeta' \Big]\\
  & \leq |D| \cdot e^{-\zeta'L^d} \cdot E^{\mathbb{Q}} \big[ \exp\{ L^d G_J(z) \} \big] \; = \; |D| \cdot e^{-\zeta'L^d} \cdot \smash{\prod_{j \in J}} E^{\mathbb{Q}} \big[ \exp\{ L^d \xi_j p_j \} \big]
  \intertext{and from \eqref{e:localclt} we get that, for $t \geq c L^2$,}
  \nonumber
  & \leq |D| \cdot e^{-\zeta' L^d} \cdot \prod_{\smash{j \in J}} \Big(1 - L^d p_j \Big)^{-1} \; \leq \; |D| \cdot \exp \Big\{ - \zeta'L^d + \sum_{\smash{j \in J}} L^d p_j + c (L^d p_j)^2 \Big\}\\
  \intertext{for $t \geq \useconstant{c:couple} L^2$ we use \eqref{e:EGJ} and \eqref{e:localclt} get}
  \nonumber
  & \leq |D| \cdot \exp \Big\{ -\zeta'L^d + \zeta L^d (1 + \frac{cL}{\sqrt{t}}) (1 + c \sup_j\{L^d p_j\}) \Big\}\\
  \nonumber
  & \leq |D| \cdot \exp \Big\{ -\zeta'L^d + \zeta L^d(1 + \frac{cL}{\sqrt{t}}) \Big\},
 \end{align}
 yielding the result.
\end{proof}

\section{Proof of the Sieving Lemma~\ref{l:couple2}}
\label{s:sieving}

The proofs in this section make use of the ``soft local time'' technique, introduced in \cite{PT12} and also used in \cite{2014arXiv1401.4498H} for a similar purpose.

Fix $q \in \{1, 2\}$ for the rest of this section and define
\begin{equation}
  \label{e:set_B}
  B = C^q_m \cap \Big(\cup_{m' \in M_k} C^2_{m'}\Big),
\end{equation}
which is the support of the sieved measure $Y_{j'}$ in Lemma~\ref{l:couple2}.

Given a collection $(x_j)_{j \in J}$, we define the stopping time for the $j$-th particle (called \emph{sieving time}) as
\begin{equation}
  \label{e:T_j}
  T_j = L_k^{2.04} \wedge \inf \Big\{ t_s; \text{ $X_{t_s} \in B$, with $s \geq 1$ integer} \Big\},
\end{equation}
where the $t_s$ stands for the equally spaced sequence $t_s = sL_k^{2.02}$, $s \geq 0$. Note that we have defined $t_0 = 0$, but it does not feature in \eqref{e:T_j}, so that $T_j$ is at least $L_k^{2.02}$. The stopping times $T_j$ that appeared in the statement of Lemma~\ref{l:couple2} are nothing more that the ones above.

Before starting the proof of Lemma~\ref{l:couple2}, let us present an auxiliary result, which encompasses the main difficulty of this section. It looks very similar to the Sieving Lemma, however we start with a very sparse collection of particles and we are allowed to dominate them by a sieved law with much higher density (which is not the case in Lemma~\ref{l:couple2}).

Below, let the sequences $L_k$, $R_k$ and $\zeta^r_k$ be as in \eqref{e:Lk}, \eqref{e:Rk} and \eqref{e:rhokr}. Fix also $q = 0, 1$.
\newconstant{c:couple_rest}
\begin{lemma}
 \label{l:couple_rest}
 Fix some $k > \useconstant{c:couple_rest}$, an index $m \in M_{k+1}$ and a collection $(x_j)_{j \in J}$ in $C_m^{q+1}$ which is $(L_k^{-2\gamma/3})$-balanced with respect to the paving $\{C_{m'}\}_{m'j \in M_k}$. Then, we can find a coupling $\mathbb{Q}$ between $\bigotimes_{j \in J} P_{x_j}$ and a sieved collection $(Y_{j'})_{j' \in J'}$ with intensity $\mu^q_m(L_k^{-\gamma/3})$, in such a way that
\begin{equation}
  \mathbb{Q} \Big[ \sum_{j \in J} \delta_{X^j_{T_j}} \leq \sum_{j' \in J'} \delta_{Y_{j'}} \Big] \geq 1 - c \exp \{ -c L_k^{0.08} \},
\end{equation}
where $T_j$ are defined as in \eqref{e:T_j}.
\end{lemma}

We postpone the proof of the above lemma to later in the text. For now, let us see how it helps proving the Sieving Lemma~\ref{l:couple2}.

\medskip

The above lemma is very useful in the process of ``sieving'', however, it is very rough in the sense that it needs a high density of particles in the sieved law for the domination to work ($L_k^{-\gamma/3} \gg L_k^{-2\gamma/3}$). We will now turn to the proof of the Sieving Lemma~\ref{l:couple2}, which was essential in the proof of Theorem~\ref{t:main}.

Recall the definitions \eqref{e:pk}, \eqref{e:rhokr} and the statement of the Sieving Lemma~\ref{l:couple2} in Section~\ref{s:renorm}.

The proof of this lemma strongly relies on the fact that the density of the $x_j$'s is slightly smaller than that of the sieved law. These two important densities to keep in mind during this proof are $\zeta_k^r > \zeta_k^{r+1}$.

\begin{proof}[Proof of the Sieving Lemma~\ref{l:couple2}]
The first thing we observe in the proof is that the stopping times defined in \eqref{e:T_j} will rarely invoke the minimum with $L_k^{2.04}$. More precisely,
\begin{equation}
  \label{e:S_empty}
  \bigotimes_{j \in J} P_{x_j} \big[ T_j = L_k^{2.04}, \text{ for some $j \in J$} \big] \leq c \exp \{ -c L_k^{0.02} \}.
\end{equation}
The bound \eqref{e:S_empty} follows from  Lemma~\ref{l:fall} as, for $k \geq c$,
\begin{equation}
  \label{e:no_exit}
  \begin{split}
    \smash{\bigotimes_{j \in J}} P_{x_j} \big[ & \text{some particle $(X^j)_{j \in J}$ exits $C_m^q$ before $L_k^{2.04}$} \big]\\
    & \qquad \leq \useconstant{c:shake} |J| \exp \Big\{ \useconstant{c:shake2} \useconstant{c:growth} \frac{L_k^{1.1}}{L_k^{(2.02 + 0.02)/2}} \Big\} \leq c |J| \exp\{ - c L_k^{0.08}\}.
  \end{split}
\end{equation}

The proof goes in two steps, we first use Lemma~\ref{l:couple} in order to couple the majority of the particles (those that have $T_j = t_1$) with a sieved law, then we apply Lemma~\ref{l:couple_rest} to deal with the remaining dust.

We first let the random walks (starting at $(x_j)_{j \in J}$) run for time $t = L_k^{2.02}$ and apply Lemma~\ref{l:couple} to couple them with a Poisson point process $\sum_{j' \in J'} \delta_{Z_{j'}}$, having intensity $(\zeta_k^q + \zeta_k^{q-1})/2 \leq 1$, in such a way that
\begin{align}
  \label{e:first_couple}
  \nonumber
  \mathbb{Q}' \Big[ \1{C_m} \; \sum_{j \in J} \delta_{X^j_t} \not \leq \sum_{j' \in J'} \delta_{Z_{j'}} \Big] & \leq L_{k+1}^d \exp \Big\{ - \big(\tfrac{\zeta_k^{q-1} - \zeta_k^q}{2} \big)
  L_k^d + c \zeta_k^q L_k^{d+1-1.01} \Big\}\\[3mm]
  & \overset{(k > c)}\leq L_{k+1}^d \exp \Big\{ - c \frac{\zeta_0 L_k}{(k+1)^2} \Big\} \overset{\eqref{e:rho_large}}\leq c \exp \Big\{ - c L_k^{1-\gamma/3} \Big\}.
\end{align}

Observe that the Poisson point process $\sum_{j'\in J'} \delta_{Z_j}$ above is not supported on the set $B = \cup_{m' \in M_k} C^0_{m'}$ as is the sieved law. This will be dealt with later by restricting the sum to the indices $J'' = \{j' \in J'; Z_{j'} \in B\}$.
When we do this, we will be left with some remaining walks (those with indices $J_1 = \{j \in J; T_j > L_k^{2.02}\}$).

The points $X^j_{L_k^{2.02}}$, for $j \in J_1$ are exactly those which have fallen in the collection of annuli $A = \cup_{m' \in M_k} C^0_{m'} \setminus C^2_{m'}$ and they should be sparse, as the fraction of points in $C^0_{m'}$ that intersects $A$ is not larger than
\begin{equation}
  \label{e:annuli_density}
  \displaystyle \frac{2d L_k^{d-1} (2R_k)}{L_k^d} = 4d\frac{R_k}{L_k} \leq 4dL_k^{-\gamma}.
\end{equation}
Therefore, $\mathbb{Q}' \big[ (Z_{j'})_{j' \in J' \setminus J''}  \text{ is not $(8d L_k^{- \gamma})$-balanced} \big] \leq c \exp\{ -c L_k^{d-\gamma} \}$ and thus
\begin{equation}
  \label{e:PoissonRing}
  \mathbb{Q}' \Big[ (X^j)_{j \in J_1}  \text{ is not $(8d L_k^{- \gamma})$-balanced} \Big] \overset{\eqref{e:first_couple}} \leq c \exp\{ -c L_k^{1/2} \}.
\end{equation}

Using the above, we will use Lemma~\ref{l:couple_rest} to couple the remaining particles (with indices in $J_1$). For $k > c$ we know that $8d L_k^{-\gamma} < L_k^{-2\gamma/3}$, so that we are in a good position to use Lemma~\ref{l:couple_rest}. On the event in \eqref{e:PoissonRing} this lemma provides us with a coupling $\mathbb{Q}''$ between $\otimes_{j \in J_1} P_{X^j_{t_1}}$ and a $(q-1)$-sieved law $\sum_{j'' \in J''} \delta_{Y_{j''}}$ with density $L_k^{-\gamma/3}$, in such a way that
\begin{equation}
  \label{e:finish_rest}
  \mathbb{Q}''\Big[ \sum_{j \in J_1} \delta_{X^j_{T^j}} \not \leq \sum_{j'' \in J''} \delta_{Y_{j''}} \Big] \leq c \exp \{ - c L_k^{0.08} \}.
\end{equation}

Observe that we need the sprinkling $L_k^{-\gamma/3}$ that we used in defining $Y_{j''}$'s to be compatible with our choice of densities \eqref{e:rhokr}. More precisely, we need
\begin{equation}
  \label{e:sprinkling_large}
  L_k^{-\gamma/3} \leq \frac{\zeta^{q-1}_k - \zeta^q_k}{2} = \frac{\zeta_0}{8(k+1)^2},
\end{equation}
Which is precisely the hypothesis on $\zeta_0$ in the statement of the lemma. The above condition is plays a central in this article, see Remark~\ref{r:pho_large} for more details.

Using \eqref{e:sprinkling_large}, we can define $\mathbb{Q}$, combining $\mathbb{Q}'$ with $\mathbb{Q}''$, in a way that
\begin{equation*}
  \begin{split}
    \mathbb{Q} \Big[ \sum_{j \in J} \delta_{X^j_{T_j}} \leq & \sum_{j'
      \in J'} \delta_{Y_{j'}}, \text{ and no $X^j$ hits $\partial
      C_m^{q-1}$ before $T_j$}
    \Big]\\
    & \geq 1 - c \; \exp \big\{ c L_k^{1-\gamma/3} \big\} - L_k^d \exp \Big\{ - c L_k^{1/2} \Big\} - c \exp \big\{ - c L_k^{0.08} \big\},
  \end{split}
\end{equation*}
where above we used Lemma~\ref{l:fall}, \eqref{e:first_couple}, \eqref{e:PoissonRing} and \eqref{e:finish_rest}. This finishes the proof of the lemma.
\end{proof}

We now give a proof of Lemma~\ref{l:couple_rest}, which was used in the proof of the Sieving Lemma~\ref{l:couple2}.

Definition \eqref{e:T_j} allows us to split the indices of particles $J_0$ into disjoint sets
\begin{equation}
  \label{e:J_n}
  F_s = \{ j \in J_0; T_{j} = t_s \}, \text{ for $s \geq 1$},
\end{equation}
corresponding to the particles that have finished sieving at time $t_s$. We also define the indices of particles that are not yet sieved at stage $s$, that is $J_0 = J$ and $J_s = J_0 \setminus \cup_{s' \leq s} F_{s'}$, for $s \geq 1$.

Lemma~\ref{l:couple} would have been enough to couple the particles at time $t_1$ with a Poisson point processes on the entire lattice. The challenge now is to make sure that the dominating Poisson point process is restricted to $B$, the so-called ``sieving''. For this we will now apply Lemma~\ref{l:couple} repeatedly $S = \lfloor L_k^{\gamma/3}/2 \rfloor$ times.

In the following proof, we let $S = \lfloor L_k^{\gamma/3}/2 \rfloor$ as above.
\begin{proof}[Proof of Lemma~\ref{l:couple_rest}]
Define, for each $s = 1, \dots, S$, an independent Poisson point process $(Z^{j'}_s)_{j' \in J'_s}$, with density $2 L_k^{-2\gamma/3}$. Clearly, if we sum these Poisson point process, we obtain another Poisson point process with density $2 L_k^{-2\gamma/3} S$. Comparing this intensity measure with that of $(Y_{j'})_{j' \in J'}$, we see that
\begin{display}
  \label{e:sum_Poisson}
  $\1{B} \cdot \sum_{s \leq S} \sum_{j' \in J'_s} \cdot \delta_{Z^{j'}_s}$ is dominated by $\sum_{j'}
  \delta_{Y_{j'}}$.
\end{display}

It is therefore clear that, in order to prove the lemma, it is enough to construct a coupling $\mathbb{Q}$ satisfying
\begin{equation}
  \label{e:XJs_Zs}
  \mathbb{Q} \Big[ \sum_{j \in F_s} \delta_{X^j_{T_j}} \leq \1{B} \cdot \sum_{j' \in J'_s} \delta_{Z^{j'}_s} \text{, for every $s \leq S$} \Big] \geq 1 - c \exp \{ -c L_k^{\gamma/3} \}.
\end{equation}
For the construction of $\mathbb{Q}$ and the proof of \eqref{e:XJs_Zs} we make repeated use of Lemma~\ref{l:couple}.

\newconstant{c:sieve1}
Using the fact that $(x_j)_{j \in J_0}$ is $(L_k^{-2\gamma/3})$-balanced, we can employ Lemma~\ref{l:couple} to obtain a coupling $\mathbb{Q}^1$ between $\otimes_{j \in J_0} P_{x_j}$ and $(Z^{j'}_1)_{j' \in J'_1}$, such that
\begin{equation}
  \label{e:Q1_good}
  \begin{split}
    \mathbb{Q}^1 & \Big[ \sum_{j \in J_0} \delta_{X^j_{t_1}} \leq \sum_{j' \in J'_{1}} \delta_{Z^{j'}_{1}} \Big] \overset{\eqref{e:no_exit}}\geq \mathbb{Q}^1 \Big[ \1{C^{q}_m} \cdot \sum_{j \in J_0} \delta_{X^j_{t_1}} \leq \sum_{j' \in J'_{1}} \delta_{Z^{j'}_{1}} \Big] - \useconstant{c:shake}j \exp\{ - c L_k^{0.08}\}\\
    & \geq 1 - |C^{q}_m| \; \exp \Big \{ - L_k^{d - 2\gamma/3} + c L_k^{d - 2\gamma/3 + 1 - 1.01} \Big \} - \useconstant{c:shake} |C^{q}_m| \exp\{ - c L_k^{0.08}\}\\
    & \geq 1 - \useconstant{c:sieve1} \; \exp \{ - c L_k^{0.08} \}.
  \end{split}
\end{equation}
Note that if we restrict the above sum over $j$'s such that $X^j_{t_1} \in B$, we obtain
\begin{equation}
  \label{e:Q1_bal}
  \mathbb{Q}^1 \Big[ \sum_{j \in F_1} \delta_{X^j_{T^j}} \leq \1{B} \cdot \sum_{j' \in J'_{1}} \delta_{Z^{j'}_{1}} \Big] \geq 1 - \useconstant{c:sieve1} \; \exp \{ - c L_k^{0.08} \},
\end{equation}
which already resembles \eqref{e:XJs_Zs}.

We are now left with the particles $(X^j_{t_1})_{j \in J_1}$, which have not been sieved at this first stage, recall that $J_1 = J_0 \setminus F_1$. We could restart the coupling from them using Lemma~\ref{l:couple} again, as long as we are on the event
\begin{equation}
  \label{e:A1}
  A_{1} := \big[\text{$(X^j_{t_{1}})_{j \in J_{1}}$ is $(L_k^{-2\gamma/3})$-balanced}\big].
\end{equation}
We expect this event to have high probability, since most particles will not fall into $C^0_{m'} \setminus C^2_{m'}$, given that
\begin{equation}
  \label{e:C2_large}
  \frac{|C^0_{m'} \setminus C^2_{m'}|}{|C^0_{m'}|} \leq 1/4, \text{ for $k > c$.}
\end{equation}

Let us now recall the whole strategy of the proof. We are constructing the coupling $\mathbb{Q}$ and for this we have first built $\mathbb{Q}^1$. In case $\mathbb{Q}^1$ fails, i.e. on the complement of the event in \eqref{e:Q1_bal}, we declare that $\mathbb{Q}$ also failed (in case of failure, we let $(X^j_t)$ and $(Y_{j'})$ be independent under $\mathbb{Q}$).

\newconstant{c:sieve2}
On the complement of $A_1$ we will also assume that $\mathbb{Q}$ failed. However, this is improbable, since
\begin{equation}
  \label{e:balanced_1}
  \begin{array}{e}
    \mathbb{Q}^1 [A_{1}^c] & \overset{\eqref{e:Q1_good}} \leq & \mathbb{Q}^1 \Big[ (Z^{j'}_1)_{\{j' \in J'_1; Z^{j'}_1 \not \in B\}} \text{ is not $(L_k^{-2\gamma/3})$-balanced} \Big] - \useconstant{c:sieve1} \exp\{ - c L_k^{0.08}\}\\
    & \overset{\eqref{e:C2_large}}\leq & \Big( \frac{L_{k+1}}{L_k} \Big)^d P \big[\text{Poisson$(L_k^{d-2\gamma/3}/2) \geq L_k^{d-2\gamma/3}$} \big] - \useconstant{c:sieve1} \exp\{ - c L_k^{0.08}\}\\
    & \leq & \useconstant{c:sieve2} \; \exp \{ - c L_k^{0.08}\}.
  \end{array}
\end{equation}

On the event $A_1$ we can proceed as above, using Lemma~\ref{l:couple} to construct $\mathbb{Q}^2$. Proceeding by induction one finally obtains, for every $s \leq S$,
\begin{align*}
  & \mathbb{Q}^s \big[{\textstyle \sum_{j \in F_s} \delta_{X^j_{T^j}} \leq \1{B} \cdot \sum_{j' \in J'_{s}} \delta_{Z^{j'}_{s}}} \big] \geq 1 - \useconstant{c:sieve1} \; \exp \{ - c L_k^{0.08} \},\\
  & A_{s} := \big[\text{$(X^j_{t_{s}})_{j \in J_{s}}$ is $(L_k^{-2\gamma/3})$-balanced}\big] \text{ and}\\
  & \mathbb{Q}^s [A_{s}^c] \geq 1 - \useconstant{c:sieve2} \; \exp \{ - c L_k^{0.08}\}.
\end{align*}
We thus construct the coupling $\mathbb{Q}$ by iterating $\mathbb{Q}^s$ for $s = 1, \dots, S$. More precisely, let $\mathbb{Q}$ be defined as follows:
\begin{itemize}
  \item if $A_s$ does not occur for some $s \leq S$ or $J_{S+1}$ is non-empty, the coupling failed, and in this case the random walks and the sieved law are independent under $\mathbb{Q}$,
  \item if $A_s$ holds for every $s \leq S$ and $J_{S+1} = \varnothing$, , $\mathbb{Q}$ is given by all the couplings $\mathbb{Q}^s$, conditionally independent given $(J_s)_{s \leq S}$ and the position of all particles $(X^j_{t_s})_{s \leq S, j \in J_s}$.
\end{itemize}
It is simple to verify that $\mathbb{Q}$ is a coupling between $\otimes_{j \in J_0} P_{x_j}$ and the Poisson point processes $(Z^j_s)_{s \leq S, j \in J'_s}$. Summing up the probabilities that $\mathbb{Q}$ failed we prove \eqref{e:XJs_Zs}, finishing the proof of Lemma~\ref{l:couple_rest}.
\end{proof}

\section{Other models}
\label{s:other_models}

In this section we comment on other models for which our techniques apply.
Although we have stated Theorems~\ref{t:absorption} and \ref{t:dd} for the process of Activated Random Walks, let us now show how the same proof can be used in other cases.

The first trivial observation is that if we show that the total activity of a certain model is stochastically dominated by that of the Activated Random Walkers, then this model will also fall in the scope of this work.
However, this simple argument is very limited.
For instance, it does not apply to the Stochastic Sandpile Model as we observe below.

\begin{remark}
  \label{r:no_domination}
  The activity of the Stochastic Sandpile Model (SSM) is not dominated by the Activated Random Walks (ARW).
  This can be easily seen by starting with $\kappa$ particles (the site's capacity) at the origin.
  For such starting configuration, the total activity for the SSM will be at least $\kappa$ as the origin will topple all its particles to random neighbors after an exponential time.
  On the other hand, for the ARW, it could happen that the first $\kappa-1$ particles jump and sleep and then the remaining particle at the origin immediately goes to sleep, resulting in total activity $\kappa - 1$.
\end{remark}

Here, we show how to extend our main results for other models.
In order to keep the exposition simple, we focus on the Stochastic Sandpile Model, but our intention is to illustrate what are the main ingredients one needs for such an extensions.
Let us first observe what are the main properties of the Activated Random Walks model that were used through the previous proofs.
We used the commutativity and monotonicity provided by Theorems~\ref{t:commutative} and \ref{t:monotonicity}, and several times the ``turning off'' of the dynamics provided by Corollary~\ref{c:offsleep}.

The proof of the above properties can often be tedious for different models, therefore we will refer to the excellent work in \cite{2013arXiv1309.3445B}, where the authors provide a collection of results that hold for a large class of models.
Following their notation, we sometimes call the particles \emph{messages}.

Let us first put some definitions in place.
Each site $x \in \mathbb{Z}^d$ will be assigned an infinite sequence of instructions $(y^x_i)_{i \geq 1}$, where $y^x_i \in \{\pm e_j; j = 1, \dots, d\}$.
Later, the $y^x_i$'s will be chosen randomly, but for now we can pick an arbitrary sequence.

The site $x$ will react to the arrival of two types of messages, corresponding to \emph{ordinary} and \emph{activation} particles.
The number of particles sent by $x$ to its neighbors will be determined by the function
\begin{equation}
f(q,r) = \min \big\{ q, \max\{q - (q \text{ mod } \kappa), r \} \big\},
\end{equation}
see Figure~\ref{m:fqr}. Note that $f(q,r)$ in non-decreasing in both arguments.

Informally speaking, after a site $x$ receives $q$ ordinary particles and $r$ activation particles, the site will launch $f(q,r)$ ordinary particles to its neighbors.
The directions to which these particles will be sent will be determined by the unitary vectors $y^x_i$.

\begin{figure}[ht]
  \centering
  \begin{tabular}{ | l | l | l | l | l | l | l | l | l }
    \vdots & \vdots & \vdots & \vdots & \vdots & \vdots & \vdots & \vdots &  \\ \hline
    0 & 1 & 2 & 3 & 4 & 5 & 6 & 7 & \dots \\ \hline
    0 & 1 & 2 & 3 & 4 & 5 & 6 & 6 & \dots \\ \hline
    0 & 1 & 2 & 3 & 4 & 5 & 6 & 6 & \dots \\ \hline
    0 & 1 & 2 & 3 & 4 & 4 & 6 & 6 & \dots \\ \hline
    0 & 1 & 2 & 3 & 3 & 3 & 6 & 6 & \dots \\ \hline
    0 & 1 & 2 & 3 & 3 & 3 & 6 & 6 & \dots \\ \hline
    0 & 1 & 1 & 3 & 3 & 3 & 6 & 6 & \dots \\ \hline
    0 & 0 & 0 & 3 & 3 & 3 & 6 & 6 & \dots \\ \hline
  \end{tabular}
  \caption{Some values of the function $f(q,r)$ when $\kappa = 3$. Here $q$ increases to the right and $r$ to the top, both starting from zero.}
\label{m:fqr}
\end{figure}

Let us now make the above description precise.
We define for $x \in \mathbb{Z}^d$,
\begin{enumerate}
\item a dictionary of messages $A_x = \{x^o, x^a\}$, corresponding to the arrival of ordinary and activation particles and
\item a state space $Q_x = \{0,1,\dots\}^2$, which stores how many particles of each type have arrived at $x$.
\end{enumerate}
We introduce the update rules $T_x: A_x \times Q_x \to Q_x$ as
\begin{equation}
  \begin{split}
    T_{(x)}(x^o, (q,r)) & = (q + 1, r)\\
    T_{(x)}(x^a, (q,r)) & = (q, r + 1),
  \end{split}
\end{equation}
which simply increases the counter of the corresponding particle type.
After updating its counter, the site $x$ will send some particles to its neighbors (only ordinary particles are sent).
The number of particles emanating from $x$ after it receives a message will be given by the increment on $f(q,r)$, that is, $f(T_{(x)}(\cdot, (q,r))) - f(q,r)$ (recall that $f$ is monotone).
And these particles will jump to directions given by the instructions $y^x_j$, where $j$ varies over fresh indices.
More precisely, for $|y| = 1$, let
\begin{equation}
T_{(x, x + y)}(\cdot, (q,r)) = (x + y)_o^l, \text{ where $l = \#\big\{f(q,r) < j \leq f\big(T_{(x)}(\cdot, (q,r))\big); y^x_j = y \big\}$}.
\end{equation}
In the above definition, one can replace $\cdot$ by either $x^o$ or $x^a$.

In order to apply the results in \cite{2013arXiv1309.3445B}, we need to prove that the rules $T_{(x)}$ and $T_{(x,x+y)}$ above are Abelian.
Indeed, independently of the order in which messages arrive, the total number of particles sent by $x$ is simply $f$ evaluated in the total number of each message type.

We can now apply Theorem~4.8 of \cite{2013arXiv1309.3445B} to conclude that the final configuration of particles as well as the total occupation field $(J_x)_{x \in \mathbb{Z}^d}$ (that is the total number of messages processed by each site) does not depend on the order in which each processor acts.

It is clear that if there are no activations messages in the system, the process will be determined simply by $f(q,0) = q - (q \text{ mod } \kappa)$, which sends $\kappa$ particles whenever $q$ reaches $\{0, \kappa, 2 \kappa, \dots\}$.
This is nothing more than the original Stochastic Sandpile Model.

Applying Lemma~4.2 of \cite{2013arXiv1309.3445B}, we conclude that adding more messages to the system (of type either $x^o$ or $x^a$) can only increase its final occupation field $(J_x)_{x \in \mathbb{Z}^d}$.

We have seen that our representation of the process provides us with analogous results to Theorems~\ref{t:commutative} and \ref{t:monotonicity}.
Let us finish by establishing a result similar to Corollary~\ref{c:offsleep}.

\begin{lemma}
  \label{c:offsleep2}
  Fix a finite collection of particles $(x_i)_{i \in I}$ and a stopping time $T_i$ for each of them. Then, for any event $A$ which is increasing in the final accumulated activity $(J_x)$,
  \begin{equation}
    \label{e:offsleep2}
    \sup \mathbb{P}_{(x_i)_{i \in I}} (A) \leq E_{(x_i)_{i \in I}} \big( \sup \mathbb{P}_{(X^i_{T_i})_{i \in I}} (A) \big),
  \end{equation}
  where the suprema are taken over all starting configurations where the state of each site $x$ is on the diagonal. As previously, the expectation above is taken with respect to an independent collection of simple random walks.
\end{lemma}

\begin{proof}
  Suppose that the system starts with a finite collection $(x^o_i)_{i \in I}$ of ordinary messages, endowed with stopping times $T_i$, that may depend on $i$.
  Moreover, each site $x \in \mathbb{Z}^d$ has state $(q_x,q_x)$ on the diagonal.
  Throughout the proof, we will manage to keep the state of every site on the diagonal, so that particles never accumulate on sites.

  Suppose that at some given time a site $x$ has at least one ordinary message $x^o$ waiting to be processed and we want this message to move to some neighboring site at random.
  What we do is send this ordinary message to be processed at $x$, together with an activation message $x^a$, so that $Q_x$ stays in the diagonal.
  Then, since $f(q,q) = q$, we are guaranteed to have the $x^o$ message jump to some random neighbor.
  After we have repeated this procedure until the stopping time $T_i$ of each ordinary message, we can stop feeding the system with activation messages, so that we are now back to the original dynamics ruled by $\mathbb{P}$ (still the state of each site remains on the diagonal).

  Since we only added activation messages to the system, our procedure can only increase the probability of the event $A$, finishing the proof of the lemma.
\end{proof}

To repeat the proof of absorption, one should only observe that Lemma~\ref{l:tiny_fraction} can be proven the same way for this new particle system (even with the supremum appearing in Lemma~\ref{c:offsleep2}).
Another important change is in the definition of $p_k(\zeta)$ in \eqref{e:pk}, that should include this supremum as well.
All other steps of the proof should be identical to what has been presented.

\appendix
\section*{Appendix}
\label{s:appendix}

\renewcommand{\theequation}{A.\arabic{equation}}
\renewcommand{\thetheorem}{A.\arabic{theorem}}
\setcounter{theorem}{0}

In this section we quote some results from \cite{PT12}, concerning how to simulate random processes using Poisson processes; this result will be used later in order to show a mixing-type result for a collection of independent random walkers.

Let $\Sigma$ be a locally compact and Polish metric space. Suppose also that we are given a measure space $(\Sigma, \mathcal{B}, \mu)$ where $\mathcal{B}$ is the Borel $\sigma$-algebra on $\Sigma$ and $\mu$ is Radon, i.e. every compact set has finite $\mu$-measure.

The above setup is standard for the construction of a Poisson point process on $\Sigma$. For this, we also consider the space of Radon point measures on $\Sigma \times \mathbb{R}_+$
\begin{equation}
 \label{e:M}
 M = \Big\{\m = \sum_{i \geq 1} \delta_{(z_i, v_i)}; z_i \in \Sigma, v_i \in \mathbb{R}_+ \text{ and } \m(K) < \infty \text{ for all $K \subseteq \Sigma$ compact} \Big\}.
\end{equation}
One can now canonically construct a Poisson point process $\m$ on the space $(M, \mathcal{M}, \mathbb{Q})$ with intensity given by $\mu \otimes \d v$, where $\d v$ is the Lebesgue measure on $\mathbb{R}_+$. For more details on this construction, see for instance~\cite{R08}, Proposition~3.6 on p.130.

The result below provides us with a way to simulate a random element of $\Sigma$ using the Poisson point process $\m$. Although this result is very intuitive, we provide here its proof for the sake of completeness and the reader's convenience.

\begin{proposition}
\label{p:simulate}
Let $g:\Sigma \to \mathbb{R}_+$ be a measurable function with $\int g(z) \mu(\d z) = 1$. For $\m = \sum_{i \geq 1} \delta_{(z_i, v_i)} \in M$, we define
\begin{equation}
 \xi = \inf \{ t \geq 0; \text{ there exists $i \geq 1$ such that $t g(z_i) \geq v_i$}\},
\end{equation}
see Figure~\ref{f:xi}. Then under the law $\mathbb{Q}$ of the Poisson point process~$\m$,
\begin{enumerate}
 \item \label{e:io} there exists a.s.\
 a single value $\hat{\imath} \geq 1$ such that
 $\xi g(z_{\hat{\imath}}) = v_{\hat{\imath}}$,
 \item \label{e:xiio} $(z_{\hat{\imath}}, \xi)$ is distributed as $g(z) \mu(dz)
\otimes \Exp(1)$,
 \item \label{e:mprime} $\m' := \sum_{i \neq \hat{\imath}} \delta_{(z_i,v_i -
\xi g(z_i))}$ has the same distribution as~$\m$ and is independent
of $(\xi, \hat{\imath})$.
\end{enumerate}
\end{proposition}

\begin{proof}
Let us first define, for any measurable $A \subset \Sigma$, the random variable
\begin{equation}
 \xi^A = \inf \{ t \geq 0; \text{ there exists $i \geq 1$ such
  that $t \mathbbm{1}_A g(z_i) \geq v_i$}\}.
\end{equation}
Elementary properties of Poisson point processes
(see for instance~(a) and~(b) in~\cite{R08}, page~130) yield that
\begin{display}
\label{e:xiA}
 $\xi^A$ is exponentially distributed (with parameter $\int_A g(z) \mu(dz)$) and if $A$ and $B$ are disjoint, $\xi^A$ and $\xi^B$ are
independent.
\end{display}

Property~(\ref{e:io}) now follows from \eqref{e:xiA}, using that~$\Sigma$ is separable and the fact that two independent exponential random variables are almost surely distinct. Observe also that
\begin{equation}
 \mathbb{Q}[\xi \geq \alpha, z_{\hat{\imath}} \in A] = \mathbb{Q}[\xi^{\Sigma \setminus A} > \xi^A \geq \alpha].
\end{equation}
Thus, using \eqref{e:xiA} we can prove property (\ref{e:xiio}) using simple properties of the minimum of independent exponential random variables.

Finally, let us establish property (\ref{e:mprime}).
We first claim that, given $\xi$, $\m'' := \sum_{i \neq \hat{\imath}} \delta_{(z_i,v_i)}$ is a Poisson point process, which is independent of $z_{\hat{\imath}}$ and, conditioned on $\xi$, has intensity measure $\mathbbm{1}_{\{v > \xi g(z)\}} \cdot \mu(\d z) \otimes \d v$.

This is a consequence of the Strong Markov property for Poisson point processes and the fact that $\{(z,v) \in \Sigma \times \mathbb{R}_+; v \leq \xi g(z)\}$ is a stopping set, see Theorem~4 of~\cite{Ros82}.

To finish the proof, we observe that, given $\xi$, $\m'$ is a mapping of $\m''$ (in the sense of Proposition~3.7 of \cite{R08}, p.134). This mapping pulls back the measure $\mathbbm{1}_{\{v > \xi g(z)\}} \cdot \mu(\d z) \otimes \d v$ to $\mu(\d z) \otimes \d v$. Noting that the latter distribution does not involve $\xi$, we conclude the proof of (\ref{e:mprime}) and therefore of the lemma.
\end{proof}

\begin{figure}
\centering \includegraphics[width = 0.65 \textwidth]{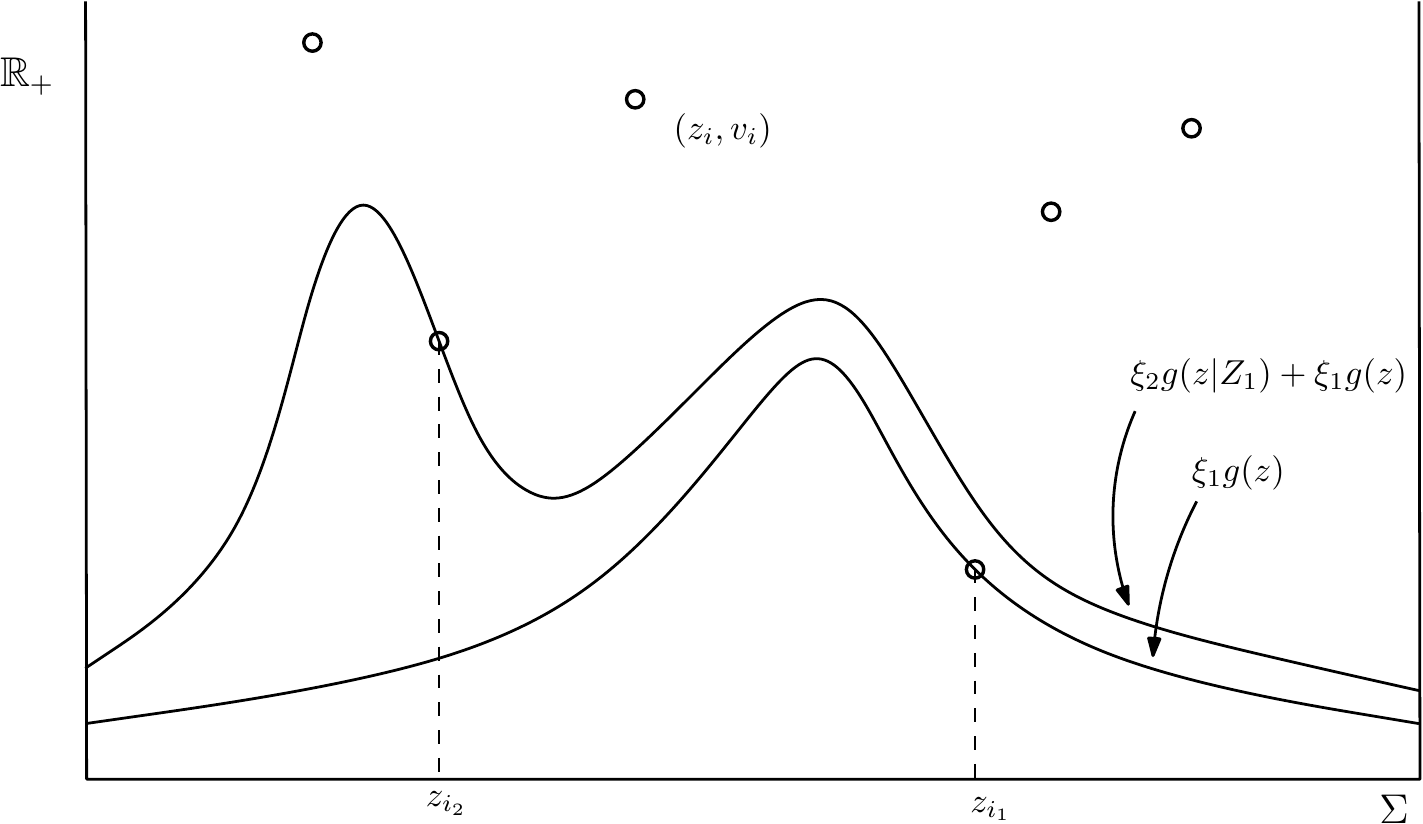}
  \caption{An example illustrating the definition of $\xi$ and $\hat{\imath}$ in Proposition~\ref{p:simulate}. More generally, $\xi_1$, $Z_1$ and $\xi_2$, $Z_2$ in \eqref{e:xis}}
  \label{f:xi}
\end{figure}

In Proposition~\ref{p:poisson} below, we use Lemma~\ref{p:simulate} in order to simulate a collection of independent random elements $Z_j$ of $\Sigma$ using the same Poisson point process $\m$ as above. Let us setup the required definitions.

Suppose that in some probability space $(M, \mathcal{M}, \mathcal{P})$ we are given an independent collection of (not necessarily identically distributed) random elements $(Z_j)_{j \geq 1}$ of $\Sigma$ such that
\begin{equation}
 \label{e:density}
 \text{for any given $j \geq 1$, the distribution of $Z_j$ is given by $g_j(z) \mu(dz)$.}
\end{equation}

In what follows, we are going to use a single Poisson point process $\m$ to simulate the above sequence $(Z_j)$.

In the same spirit of the definition of $\xi$ in Proposition~\ref{p:simulate}, we will now define what we call the \emph{soft local time}~$G$:
\begin{equation}
 \label{e:xis}
 \begin{split}
  & \xi_{1} = \inf \big\{ t \geq 0; \text{ there exists $i \geq 1$ such that $tg_1(z_i) \geq v_i$}\big\},\\
  & \quad G_{1}(z) = \xi_{1} \; g_1(z),\\
  & \qquad \quad \vdots\\
  & \xi_{k} =  \inf \big\{ t \geq 0; \text{ for at least $k$ values of $i$, $tg_j(z_i) + G_{k-1}(z_i) \geq v_i$}\big\},\\
  & \quad G_{k}(z) = \xi_{1} \; g_1(z) + \dots + \xi_{k} \; g_k(z),
 \end{split}
\end{equation}
see Figure~\ref{f:xi} for an illustration of this procedure.

Applying Proposition~\ref{p:simulate} repeatedly, we conclude by induction in~$J$ that

\begin{proposition}
 For an independent sequence $(Z_j)_{j \geq 1}$ as in \eqref{e:density} and with the definitions in \eqref{e:xis}, we have that
 \label{p:poisson}
 \begin{align}
  \label{e:ijk}
   & (\xi_{k})_{j = 1}^{J} \text{ are i.i.d. Exp$(1)$-random variables.}\\
   & \text{there is a.s. a unique $i_{J}$ such that $G_{J}(z_{i_{J}}) = v_{i_J}$}\\
   & (z_{i_{1}}, \dots, z_{i_{J}}) \overset d\sim (Z_1, \dots, Z_{J}) \text{ and}\\
   & \text{$\m' := \;\; \sum_{\mathclap{i \not \in \{i_{1}, \dots, i_{J}\}}} \;\; \delta_{(z_i, v_i - G_{J}(z_i))}$ is distributed as $\m$ and independent of the above.}
 \end{align}
\end{proposition}

We would like to finish this section giving a flavor of how the above proposition can help with the proof of Lemma~\ref{l:couple}. The next corollary shows that performing the above construction for two collections (say $Z_k$ and ${Z'}_k$) of independent elements of $\Sigma$ while using the same Poisson point process as basis, can provide a powerful coupling between them. More precisely,

\begin{corollary}
 \label{c:couplesystem}
 Suppose we are given a family of densities $(g_j(\cdot))_{j = 1}^J$ and the corresponding $\xi_{j}$, $G_{j}$ and $i_{j}$, for $j = 1, \dots, J$, as in \eqref{e:xis} and \eqref{e:ijk}. Then, for any $\zeta > 0$,
 \begin{equation*}
  \mathbb{Q} \Big[ \sum_{j \leq J} \delta_{z_{i_{j}}} \leq \sum_{i; v_i < \zeta} \delta_{z_i}, \Big]
   \geq \mathbb{Q} \Big[ G_{J} \leq \zeta \Big].
 \end{equation*}
\end{corollary}

Note that the right-hand side of the above bound only depends on the soft local time, which may be estimated for instance, through large deviation bounds.

\bibliographystyle{plain}
\bibliography{../BibTeX/all}

\end{document}